\documentclass[a4paper]{article}
\usepackage[dvips]{graphicx}
\usepackage{amssymb}
\usepackage{amsmath}
\usepackage{mathrsfs}
\usepackage{color}
\usepackage{amsthm}
\usepackage{multirow}
\DeclareMathAlphabet{\mathpzc}{OT1}{pzc}{m}{it}
\newcommand{\be}{\begin{equation}}
\newcommand{\ee}{\end{equation}}

\newcommand{\R}{\mathbb{R}}

\newcommand{\Mag}{\mathcal{M}^\alpha(\gamma)}
\newcommand{\Manp}{\mathcal{M}^\alpha_n(p_n)}

\newcommand{\FRAC}[2]{\leavevmode\kern.1em\raise.5ex\hbox{\the\scriptfont0 #1} }
\theoremstyle{definition}
\newtheorem{df}{Definition}[section]
\newtheorem{rem}{Remark}[section]
\newtheorem*{syb*}{Symbol}
\theoremstyle{plain}
\newtheorem{thm}{Theorem}[section]
\newtheorem{prop}{Proposition}[section]
\newtheorem{lem}{Lemma}[section]
\newtheorem{cor}{Corollary}[section]

\makeatletter
    
    \@addtoreset{equation}{section}
  \makeatother

\title{A discretization of O'Hara's knot energy and its convergence}
\author{Shoya Kawakami}
\date{\today}

\begin{document}

\maketitle

%
%

\begin{abstract}

In this paper, we propose a discrete version of O'Hara's knot energy defined on polygons embedded in the Euclid space.
It is shown that values of the discrete energy of polygons inscribing the curve which has bounded O'Hara's energy converge to the value of O'Hara's energy of its curve.
Also, it is proved that the discrete energy converges to O'Hara's energy in the sense of $\Gamma$-convergence.
Since $\Gamma$-convergence relates to minimizers of a functional and discrete functionals, we need to investigate the minimality of the discrete energy.

\end{abstract}

%
%

\section{Introduction}

Let
$ \gamma : \mathbb{S}_L \to \mathbb{R}^d $
be a closed curve in
$ \mathbb{R}^d $
for $ L>0 $ and $ d \geq 2$,
where
$\mathbb{S}_L$
is a circle with length
$L$.
The curve
$ \gamma $
is said to be a \textit{knot} when it is embedded in
$ \mathbb{R}^3 $.
For
$\alpha$,
$q \in (0,\infty)$,
O'Hara's knot energies of $ \gamma $
are denoted by $ \mathcal{E}^{\alpha ,q} (\gamma) $
and are defined by
\[
	\mathcal{E}^{\alpha ,q} (\gamma)
	:=
	\frac 1 \alpha
	L^{\alpha q-2}
	\int_{\mathbb{S}_L} \int_{\mathbb{S}_L}
	\left(
	\mathcal{M}^\alpha (\gamma)
	\right)^q
	dsdt,
\]
where
\[
	\mathcal{M}^\alpha (\gamma)
	=
	\mathcal{M}^\alpha (\gamma)(s,t)
	:=
	\frac{1}{ | \gamma(t) - \gamma(s) |^\alpha }
	-
	\frac{1}{ D( \gamma(s) , \gamma(t) )^\alpha },
\]
and
$ D( \gamma(s) , \gamma(t) ) $
is the intrinsic distance between
$ \gamma(s) $
and
$ \gamma(t) $.
Note that the coefficient
$L^{\alpha q-2}$
ensures O'Hara's energies are scale invariant.
These energies were introduced by J. O'Hara \cite{O91} to give an answer to the question,
``What is the most beautiful knot in a given knot class ?''.
Therefore, O'Hara's energies were constructed so that as the knot becomes more well-balanced,
the value of the energy decreases.
Also, when we deform a knot, it is not desirable that the knot class to which the knot belongs changes.
Thus, these energies were also constructed so that divergence occurs if a knot has self-intersection.

However, it is difficult to calculate values of O'Hara's energies directly,
and as a result,
it is not easy to evaluate well-balancedness.
Therefore,
it is desirable to numerically calculate these energies.
A discretization of O'Hara's energy with
$ \alpha = 2 $,
$q=1 $
was proposed by Kim-Kusner \cite{KK}.
Let
$ p_n : \mathbb{S}_L \to \mathbb{R}^d $
be a polygon with
$n$
edges parametrized by arc-length and embedded in
$ \mathbb{R}^d $
with length
$L$.
Let
$ a_i $
be the value of the arc-length parameter at the $i$-th vertex of
$p_n$,
and note that
$ p_n $ is made by connecting
$ \{ p_n(a_i) \} $
in turn.
Then, the polygonal discrete energy,
denoted by
$ \mathcal{E}_n (p_n) $,
is defined by
\[
	\mathcal{E}_n (p_n)
	:=
	\frac 1 2
	\sum_{\substack{i,j=1\\i\ne j}}^n
	\mathcal{M}_n (p_n)
	| p_n(a_{i+1}) - p_n(a_i) |
	| p_n(a_{j+1}) - p_n(a_j) |,
\]
where
\[
	\mathcal{M}_n (p_n)
	:=
	\frac{1}{ | p_n(a_j) - p_n(a_i) |^2 }
	-
	\frac{1}{ D( p_n(a_j) , p_n(a_i) )^2 }.
\]
Using this discrete energy,
Kim-Kusner \cite{KK} calculated values of O'Hara's energy with
$ \alpha = 2 $, $q=1$
of torus knots by numerical experiments.

Scholtes \cite{Sch} addressed to what extent
$ \mathcal{E}_n $
approximates
$ \mathcal{E}^{2,1} $.
For a closed curve
$ \gamma : \mathbb{S}_L \to \mathbb{R}^d $,
\textit{inscribed polygons} in
$ \gamma $
were considered.
Let
$ p_n $
be an inscribed polygon,
and suppose the vertices correspond to parameters
$ b_j \in \mathbb{S}_L $;
that is, $p_n$ is made by connecting
$ \{ \gamma (b_j) \} $
in turn.
It was shown that if
$ \gamma $
belongs to
$ C^{1,1} (\mathbb{S}_L , \mathbb{R}^d) $
and that there exists
$ c,\bar{c} >0 $
such that
\[
	\frac c n \leq \min_{k=1,\ldots n} | \gamma(b_{k+1}) - \gamma(b_k) |
	\leq \max_{k=1,\ldots n} | \gamma(b_{k+1}) - \gamma(b_k) |
	\leq \frac{\bar{c}}{n},
\]
then it holds that for all
$\varepsilon >0$,
there exists
$ C_\varepsilon >0 $
depending on
$ \gamma $,
$c$,
and
$\bar{c}$
such that
\[
	| \mathcal{E}^{2,1} (\gamma) - \mathcal{E}_n (p_n) |
	\leq C_\varepsilon \frac 1 {n^{ 1- \varepsilon }}.
\]
Also, it was shown that if
$ \gamma \in C^{0,1} (\mathbb{S}_L , \mathbb{R}^d) $
and
$ \mathcal{E}^{2,1} (\gamma) < \infty $,
then it holds that
\[
	\lim_{n \to \infty} \mathcal{E}_n (p_n) = \mathcal{E}^{2,1} (\gamma).
\]
In addition, the idea of \textit{$\Gamma$-convergence} was used in \cite{Sch}.
$\Gamma$-convergence was introduced by De Giorgi and is one type of convergence of a sequence of functionals which is very useful when we study the convergence of the sequence of minimal values of each functional to those to the limit functional.
In \cite{Sch}, it was shown that
$ \mathcal{E}_n $
converges to
$ \mathcal{E}^{2,1} $
in the sense of $\Gamma$-convergence on metric spaces.
Here, these metric spaces contain
$C^1$
curves and equilateral polygons with length
$1$
belonging to a given tame knot class equipped with the metric induced by
$L^r$-norm and
$W^{1,r}$-norm with
$ r \in [1, \infty] $.
Using this,
it was shown that
minimal values of
$ \mathcal{E}_n $
converge to the minimal value of
$ \mathcal{E}^{2,1} $.
Moreover,
it was shown that minimizers of
$ \mathcal{E}_n $
in the set of equilateral polygons
are regular polygons
and that the minimizers are unique except congruent transformations and similar transformations.

$\mathcal{E}^{2,1}$
is called the M\"{o}bius energy,
since this energy is invariant under M\"{o}bius transformations.
Scholtes did not use this property for proving his result,
and thus it is natural to believe that this argument may be applicable to all of O'Hara's energies;
we prove this here.
More precisely, in this article,
we propose a discretization of
$(\alpha ,q)$-O'Hara energies by using the idea of \cite{Sch},
and we discuss approximation of the discrete energies to O'Hara energies and the $\Gamma$-convergence.

\begin{df}[A discretization of $(\alpha ,q)$-O'Hara energies]
Let
$\alpha\ ,q \in (0, \infty)$,
and let
$ p_n : \mathbb{S}_L \to \mathbb{R}^d $
be a polygon parametrized by arc-length with
$n$
vertices whose total length is
$L>0$.
Let
$a_j$
be the value of arc-length parameters corresponding to its vertices and assume
\[
	0 \leq a_1 < a_2 < \cdots < a_n < L\ \ (\operatorname{mod}\,L).
\]
Then, we define
$\mathcal{E}_n^{\alpha ,q} (p_n)$
by
\[
	\mathcal{E}_n^{\alpha ,q}(p_n)
	:=
	\frac 1 \alpha
	L^{\alpha q-2}
	\sum_{\substack{i,j=1\\i\ne j}}^n
	( \mathcal{M}^\alpha_n (p_n) )^q
	| p_n(a_{i+1}) - p_n(a_i) |
	| p_n(a_{j+1}) - p_n(a_j) |,
\]
where
\[
	\mathcal{M}^\alpha_n (p_n)
	=
	\mathcal{M}^\alpha_n (p_n)(a_i,a_j)
	:=
	\frac{1}{ | p_n(a_j) - p_n(a_i) |^\alpha }
	-
	\frac{1}{ D( p_n(a_j) , p_n(a_i) )^\alpha }.
\]
\end{df}

Our main theorems are as follows.

\begin{thm}[cf.\ Theorems \ref{orderinscribed} and \ref{inscribed}]
Assume that
$ \alpha \in ( 0 , \infty ) $
and
$ q \in [1, \infty ) $
satisfy
$ 2 \leq \alpha q < 2q+1 $,
and set
$\displaystyle \sigma := \frac{\alpha q-1}{2q}$.
\begin{enumerate}
\item\label{1}
Let
$ \gamma \in C^{ 1,1 } ( \mathbb{S}_L , \mathbb{R}^d ) $
be a curve parametrized by arc-length embedded in
$ \mathbb{R}^d $,
where
$L$
is the length of
$\gamma$.
Let
$ c $, $ \bar{c} > 0 $,
and set
$
	K := 
	\| \gamma^{ \prime \prime } \|
	_{ L^\infty ( \mathbb{S}_L , \mathbb{R}^d ) }
$.
In addition,
for
$ n \in \mathbb{N} $,
let
$ \{ b_k \}_{ k=1 }^n $
be a division of
$ \mathbb{S}_L $
satisfying
\begin{equation*}
	\frac{ c L }{ n }
	\leq
	\min_{k=1, \ldots ,n}
	| \gamma ( b_{ k+1 } ) - \gamma( b_k ) |
	\leq
	\max_{k=1, \ldots ,n}
	| \gamma ( b_{ k+1 } ) - \gamma( b_k ) |
	\leq
	\frac{ \bar{c} L }{ n },
	\label{div.bk}
\end{equation*}
and let $p_n$ be the inscribed polygon in $ \gamma $ with vertices
$ \gamma(b_1) , \ldots , \gamma(b_n) $,
where
we extend the notation
$ \gamma (b_k) $
to all
$ k \in \mathbb{Z} $
in the natural way via congruency modulo
$n$;
i.e.,
$\gamma(b_0) = \gamma(b_n)$,
$\gamma(b_1) = \gamma(b_{n+1})$,
and so on.
Then,
if the number $n$ of points of the division is sufficiently large,
there exists
$ C > 0 $
depending on 
$ c $, $ \bar{c} $, $ \mathcal{E}^{\alpha,q} (\gamma) $
such that
\[
	| \mathcal{E}^{ \alpha , q } ( \gamma ) 
	- \mathcal{E}_n^{ \alpha , q } ( p_n )|
	\leq
	C
	\{
	( LK )^{2q}
	+
	( LK )^{2q+2}
	\}
	\cfrac{1}{ n^{2q- \alpha q+1} }.
\]
Furthermore, if $\alpha \leq 2$,
then there exists $C>0$ depending on $ c $, $ \bar{c} $, $ \mathcal{E}^{\alpha,q} (\gamma) $
such that
\[
	| \mathcal{E}^{ \alpha , q } ( \gamma ) 
	- \mathcal{E}_n^{ \alpha , q } ( p_n )|
	\leq
	C
	\{
	( LK )^{ \alpha q }
	+
	( LK )^{ \alpha q - \alpha + 2 }
	+
	( LK )^{ \alpha q + 2 }
	\}
	\cfrac{ \log n }{ n }.
\]
\item
Let
$ \gamma \in W^{ 1+\sigma , 2q } (\mathbb{S}_L,\mathbb{R}^d) $,
and let $ p_n $ be the inscribed polygon as in 1.
Then, we have
\[
	\lim_{ n \rightarrow \infty }
	\mathcal{E}_n^{\alpha ,q} ( p_n )
	=
	\mathcal{E}^{\alpha ,q} ( \gamma ).
\]
\end{enumerate}
\end{thm}

\begin{thm}[cf. Theorem \ref{gamma}]
For
$\alpha \in (0,\infty)$,\ $ q \in [1,\infty) $
satisfying
$ 2 \leq \alpha q < 2q+1 $,
$\mathcal{E}_n^{\alpha ,q}$
converges to
$\mathcal{E}^{\alpha ,q}$
in the sense of
$\Gamma$-convergence on a metric space $X$.
\end{thm}

\begin{rem}
\begin{enumerate}
\item
$W^{1+\sigma , 2q}$
is called the \textit{Sobolev-Slobodeckij space},
and it was used in \cite{Bla} to give a necessary and sufficient condition that O'Hara's energies are bounded.
\item
A metric function on $X$, $d_X : X \times X \to \mathbb{R} $, satisfies
\[
	C_1 \| f-g \|_{L^1 (\mathbb{S}_L , \mathbb{R}^d)} \leq d_X (f,g) \leq
	C_2 \| f-g \|_{W^{1,\infty} (\mathbb{S}_L , \mathbb{R}^d)}
\]
for $f$, $g \in X$,
where
$C_1$,
$C_2 > 0$
are constants.
The full definition of $X$ is given in Section \ref{sec.gamma}.
\end{enumerate}
\end{rem}

In addition,
we discuss minimizers of the discrete energies $\mathcal{E}_n^{\alpha ,q}$ of the set of all  \textit{equilateral} polygons with $n$ edges.
If we try to decrease the values of the discrete energies $\mathcal{E}_n^{\alpha ,q}$ without conditions of lengths of edges and the numbers of vertices,
polygons degenerate into triangles.
Hence, the infimum of the discrete energies $\mathcal{E}_n^{\alpha ,q}$ of the set of all polygons is $0$.
That is reason why we consider their minimizers in the set of all equilateral polygons with $n$ edges.

\begin{thm}[cf. Theorem \ref{discrete--minimizer}]
Let
$ \alpha \in (0, \infty ) $
and
$ q \in [1, \infty ) $.
Then, minimizers of
$ \mathcal{E}_n^{\alpha ,q} $
are regular polygons in the set of equilateral polygons with $n$ edges.
In particular, a regular polygon with $n$ edges is the only minimizer,
except for congruent transformations and similar transformations.
\end{thm}

In what follows,
for simplicity,
we write
$ D( \gamma(s), \gamma(t) )$,
$ D(p_n(a_i) , p_n(a_j)) $
as
$|t-s|$,
$|a_j - a_i|$
respectively.

\paragraph{Acknowledgment.}
The author is grateful to Professor Takeyuki Nagasawa for his direction and many useful advices and remarks.
Additionally,
the author would like to thank Professor Neal Bez for English language editing and mathematical comments.
%
%

\section{Approximation of O'Hara's energy by inscribed polygons}

In this section, we show that the discrete energy defined in previous section converges to O'Hara's energy under certain conditions.

First, in order to describe our claim, we define the \textit{Sobolev-Slobodeckij space}.

\begin{df}[The (cyclic) Sobolev-Slobodeckij space]
Let
$ \sigma \in (0,1) $, 
and let
$ q \in [1,\infty) $.
We define the \textit{Sobolev-Slobodeckij space} by
\[
	W^{\sigma ,q} (\mathbb{S}_L,\mathbb{R}^d)
	:=
	\left\{
	f \in L^q (\mathbb{S}_L,\mathbb{R}^d)
	\,\left|\,
	\int_{\mathbb{S}_L} \int_{\mathbb{S}_L}
	\frac{ | f^\prime (t) - f^\prime (s) |^q }{ |t-s|^{ 1+q \sigma } }
	dsdt
	< \infty
	\right.
	\right\},
\]
equipped with the norm
\[
	\| f \|_{W^{\sigma ,q} (\mathbb{S}_L,\mathbb{R}^d)}
	:=
	\| f \|_{L^q (\mathbb{S}_L,\mathbb{R}^d) }
	+ 
	\left(
	\int_{\mathbb{S}_L} \int_{\mathbb{S}_L}
	\frac{ | f^\prime (t) - f^\prime (s) |^q }{ |t-s|^{ 1+q \sigma } }
	dsdt
	\right)^{1/q}.
\]
Furthermore, we put
\[
	W^{ 1+\sigma , q } (\mathbb{S}_L,\mathbb{R}^d)
	:=
	\{
	f \in W^{ 1 , q } (\mathbb{S}_L,\mathbb{R}^d)
	\, | \,
	f' \in W^{ \sigma , q } (\mathbb{S}_L,\mathbb{R}^d)
	\}.
\]
\end{df}

Using the Sobolev-Slobodeckij space, we can describe the necessary and sufficient conditions for the boundedness of O'Hara's energy.

\begin{prop}[{\cite[Theorem 1.1]{Bla}}]\label{propbdd}
Let
$ \gamma \in C^{ 0 , 1 } (\mathbb{S}_L,\mathbb{R}^d) $
be a regular curve.
Let
$ \alpha \in (0, \infty ) $
and
$ q \in [1, \infty ) $
with
$ 2 \leq \alpha q < 2q+1 $,
and set
$ \displaystyle \sigma := \frac{ \alpha q - 1 }{ 2q } $.
Then,
$ \mathcal{E}^{ \alpha , q } ( \gamma ) < \infty $
if and only if
$ \gamma \in W^{ 1+\sigma , 2q } (\mathbb{S}_L,\mathbb{R}^d) $.
\end{prop}

From now on,
we write
$ \sigma = ( \alpha q - 1 )/( 2q ) $.
For a given regular curve $\gamma$,
we say that a polygon $p$ is \textit{inscribed} in $\gamma$
if $p$ satisfies
\begin{itemize}
	\item[(i)]
	the number of vertices is finite,
	\item[(ii)]
	the set of vertices is
	$ \{ \gamma(s_1) , \gamma(s_2) , \ldots , \gamma(s_n) \} $
	with
	$ s_1 < s_2 < \cdots < s_n \ (< s_1+L) $,
	\item[(iii)]
	the $i$-th edge is the segment jointing
	$ \gamma(s_i) $ and $ \gamma(s_{i+1}) $, where we interpret
	$ s_{n+1} = s_1 $.
\end{itemize}

The aim of this section is to prove the following two theorems.

\begin{thm}
[The rate of convergence of discretization via the approximation by inscribed polygons]
\label{orderinscribed}
Assume that
$ \alpha \in ( 0 , \infty ) $
and
$ q \in [1, \infty ) $
satisfy
$ 2 \leq \alpha q < 2q+1 $.
Let
$ \gamma \in C^{ 1,1 } ( \mathbb{S}_L , \mathbb{R}^d ) $
be a curve parametrized by arc-length embedded in
$ \mathbb{R}^d $,
where $L$ is the length of $\gamma$.
Let
$ c $, $ \bar{c} > 0 $,
and set
$
	K := 
	\| \gamma^{ \prime \prime } \|
	_{ L^\infty ( \mathbb{S}_L , \mathbb{R}^d ) }
$.

In addition,
for
$ n \in \mathbb{N} $,
let
$ \{ b_k \}_{ k=1 }^n $
be a division of\,\,\,$ \mathbb{S}_L $
satisfying
\begin{equation}
	\frac{ c L }{ n }
	\leq
	\min_{k=1, \ldots ,n}
	| \gamma ( b_{ k+1 } ) - \gamma( b_k ) |
	\leq
	\max_{k=1, \ldots ,n}
	| \gamma ( b_{ k+1 } ) - \gamma( b_k ) |
	\leq
	\frac{ \bar{c} L }{ n },
	\label{div.bk}
\end{equation}
and let $p_n$ be the inscribed polygon in $ \gamma $ with vertices
$ \gamma(b_1) , \ldots , \gamma(b_n) $.
Then,
if the number $n$ of points of the division is sufficiently large,
there exists
$ C > 0 $
depending on 
$ c $, $ \bar{c} $, $ \mathcal{E}^{\alpha,q} (\gamma) $
such that
\[
	| \mathcal{E}^{ \alpha , q } ( \gamma ) 
	- \mathcal{E}_n^{ \alpha , q } ( p_n )|
	\leq
	C
	\{
	( LK )^{2q}
	+
	( LK )^{2q+2}
	\}
	\cfrac{1}{ n^{2q- \alpha q+1} }.
\]
Furthermore, if $\alpha \leq 2$,
then there exists $C>0$ depending on $ c $, $ \bar{c} $, $ \mathcal{E}^{\alpha,q} (\gamma) $
such that
\[
	| \mathcal{E}^{ \alpha , q } ( \gamma ) 
	- \mathcal{E}_n^{ \alpha , q } ( p_n )|
	\leq
	C
	\{
	( LK )^{ \alpha q }
	+
	( LK )^{ \alpha q - \alpha + 2 }
	+
	( LK )^{ \alpha q + 2 }
	\}
	\cfrac{ \log n }{ n }.
\]
\end{thm}

\begin{thm}[The convergence of the discrete energy of inscribed polygons]\label{inscribed}
Assume that
$ \alpha \in ( 0 , \infty ) $
and
$ q \in [1, \infty ) $
satisfy
$ 2 \leq \alpha q < 2q+1 $.
Let
$ \gamma \in W^{ 1+\sigma , 2q } (\mathbb{S}_L,\mathbb{R}^d) $,
and let $ p_n $ be the inscribed polygon as in Theorem \ref{orderinscribed}.
Then, we have
\[
	\lim_{ n \rightarrow \infty }
	\mathcal{E}_n^{\alpha ,q} ( p_n )
	=
	\mathcal{E}^{\alpha ,q} ( \gamma ).
\]
\end{thm}

\begin{rem}
Since it holds that
\[
	C^{1,1} (\mathbb{S}_L,\mathbb{R}^d)
	\subset
	W^{ 1+\sigma , 2q } (\mathbb{S}_L,\mathbb{R}^d),
\]
then
$\gamma$ in Theorem \ref{orderinscribed} has always bounded energy,
i.e., $\mathcal{E}^{\alpha ,q} (\gamma) < \infty$.
\end{rem}


\subsection{Lemmas}

In this subsection,
we prove estimates and properties of parameters of curves and polygons in preparation for our proofs of Theorems \ref{orderinscribed} and \ref{inscribed}.

First,
we observe the \textit{bi-Lipschitz continuity} property of curves with bounded energy.

\begin{lem}[$\mbox{\cite[Lemma 2.1]{Bla}}$]
Let
$ \gamma \in C^{0,1} ( \mathbb{S}_L , \mathbb{R}^d ) $
satisfy
$ \mathcal{E}^{\alpha ,q} (\gamma) < \infty $.
Then,
there exists
$ C_b > 0 $
such that
\begin{equation}
	| t-s | \leq C_b | \gamma (t) - \gamma (s) |
	\label{bi-Lip}
\end{equation}
for
$ s $,
$ t \in \mathbb{S}_L $.
\end{lem}

Next, we give the parametrization of an inscribed polygon.
For a division
$ \{ b_k \}_{k=1}^n $
on
$ \mathbb{S}_L $,
let
$ p_n $
be the inscribed polygon in $\gamma$
with vertices
$ \gamma (b_k) \ ( k=1 , \ldots ,n ) $.
We extend the notation
$ \gamma (b_k) $
to all
$ k \in \mathbb{Z} $
in the natural way via congruency modulo
$n$;
i.e.,
$\gamma(b_0) = \gamma(b_n)$,
$\gamma(b_1) = \gamma(b_{n+1})$,
and so on.
Let
$
	\tilde{L}_n
	= \sum_{k=1}^n 
	| \gamma ( b_{k+1} ) - \gamma ( b_k ) |
$
be the length of
$p_n$.
Set
$
	a_i = \sum_{k=0}^{i-1} 
	| \gamma ( b_{k+1} ) - \gamma ( b_k ) |
$
as the value of the arc-length parameter of the $i$-th vertex of $ p_n $.
Then,
note that
\begin{equation}
	| b_j - b_i |
	\geq | a_j - a_i |
	\geq
	| \gamma ( b_j ) - \gamma ( b_i ) |
	\geq C_b^{ -1 } | b_j - b_i |.
	\label{ab--bi.Lip.}
\end{equation}

In what follows,
we set
$
	\displaystyle
	N:= 4 C_b \frac{ \bar{c} }{ c }
$.
We get the following lemma by the triangle inequality.

\begin{lem}
Let
$ t \in [ b_j , b_{j+1} ] $,
$ s \in [ b_i , b_{i+1} ] $.
Then, we have
\begin{equation}
	| t - s |
	\leq 
	\left( 
	1 + 2C_b \frac{ \bar{c} }{ c }
	\right)
	| b_j - b_i |.
	\label{tsleqbjbi}
\end{equation}
In addition,
if
$ | j - i | \geq N $,
we have
\begin{equation}
	| t - s | 
	\geq
	C_b^{ -1 } \frac{ c }{ 2\bar{c} } | b_j - b_i |.
	\label{tsgeqbjbi}
\end{equation}
\end{lem}

In the next lemma,
we calculate the difference between the arc-length and the distance of two points.

\begin{lem}
\begin{enumerate}
\item
Let
$ \gamma \in C^{1,1} ( \mathbb{S}_L , \mathbb{R}^d ) $,
and
$ s $, $ t \in \mathbb{S}_L $.
Then, we have
\begin{equation}
	0 \leq | t-s |^2 - | \gamma (t) - \gamma (s) |^2
	\leq
	\frac{K^2}{2} | t-s |^4.
	\label{ts--est.}
\end{equation}
\item
Let
$ q \in [1, \infty ) $,
$ \gamma \in C^{0,1} ( \mathbb{S}_L , \mathbb{R}^d ) $
and let
$ s $,
$ t \in \mathbb{S}_L $.
Then, we have
\begin{eqnarray}
	\lefteqn{
	| t-s | - | \gamma (t) - \gamma (s) |
	}
\nonumber\\
	& \leq &
	\frac{1}{2}
	| t-s |^{ \alpha + 1 - 2/q }
	\left(
	\int_s^t \int_s^t
	\frac
	{ | \gamma ' (v) - \gamma ' (u) |^{ 2q } }
	{ | v-u |^{ \alpha q } }  dudv
	\right)^{ 1/q }.
	\label{ts--int.est.}
\end{eqnarray}
\end{enumerate}
\end{lem}

\begin{proof}
We only prove (\ref{ts--int.est.}).
In the case where
$ q=1 $,
we get
\begin{align*}
	| t-s | - | \gamma(t) - \gamma(s) |
	& \leq \,
	\frac{ | t-s |^2 - | \gamma(t) - \gamma(s) |^2 }{ | t-s | }
\\
	& = \,
	\frac{ 1 }{ 2 | t-s | }
	\int_s^t \int_s^t 
	| \gamma ' (v) - \gamma ' (u) |^2  dudv
\\
	& \leq \,
	\frac{ 1 }{ 2 } | t-s |^{\alpha -1}
	\int_s^t \int_s^t 
	\frac{| \gamma ' (v) - \gamma ' (u) |^2}{| v-u |^\alpha}
	dudv.
\end{align*}
On the other hand,
in the case where
$ q \in (1, \infty ) $,
we get
\begin{align*}
	&\ | t-s | - | \gamma(t) - \gamma(s) |
\\
	\leq & \ 
	\frac{ 1 }{ 2 | t-s | }
	\left(
	\int_s^t \int_s^t 
	\frac{ | \gamma ' (v) - \gamma ' (u) |^{2q} }{ |v-u|^{\alpha q} } dudv	
	\right)
	^{1/q}
	\left(
	\int_s^t \int_s^t
	| v-u |^{ \frac{ \alpha q}{ q-1} } dudv
	\right)^{1-1/q}
\\
	\leq &\ 
	\frac{ 1 }{ 2 }
	| t-s |^{\alpha + 1 - 2/q}
	\left(
	\int_s^t \int_s^t 
	\frac{ | \gamma ' (v) - \gamma ' (u) |^{2q} }{ |v-u|^{\alpha q} } dudv	
	\right)
	^{1/q}
\end{align*}
by H\"{o}lder's inequality.
\end{proof}

The following lemma is proved by simple calculations, hence, we omit the proof.

\begin{lem}

\begin{enumerate}
	\item Let $ 0 < \alpha \leq 2 $.
	Then we have
	\begin{equation}
		1 -x^\alpha \leq (1-x^2)^{ \alpha /2}
		\label{alpha2less1-xest.}
	\end{equation}
	for all $ 0 \leq x \leq 1 $.
	\item Let $ a > 0 $.
	Then, we have
	\begin{equation}
		1 - x^a \leq ( a+1 )( 1-x ).
	\label{1-xest.}
	\end{equation}
	for all $ 0 \leq x \leq 1 $.
\end{enumerate}

\end{lem}

Finally,
we have the following lemma,
which may be proved by using (\ref{ts--est.}), (\ref{alpha2less1-xest.}), and (\ref{1-xest.}).

\begin{lem}

\begin{enumerate}
\item
Let
$ \alpha > 0$.
Then we have
\begin{equation}
	| t-s |^\alpha - | \gamma (t) - \gamma (s) |^\alpha
	\leq
	\left( \frac \alpha 2 +1 \right)
	\frac{ K^2 }{ 2 } | t-s |^{ \alpha +2}
	\label{alpha2more--est.}
\end{equation}
for all
$s$,
$t \in \mathbb{S}_L$.
\item
Let
$0 < \alpha \leq 2$.
Then we have
\begin{equation}
	| t-s |^\alpha - | \gamma (t) - \gamma (s) |^\alpha
	\leq
	\frac{ K^{ \alpha } }{ 2^{ \alpha/2 }} | t-s |^{ 2\alpha }
	\label{alpha2less--est.}
\end{equation}
for all
$s$,
$t \in \mathbb{S}_L$.
\end{enumerate}
\end{lem}

In subsections \ref{pforderinscribed} and \ref{pfinscribed},
unless otherwise noted,
we assume that
$\alpha \in (0,\infty)$
and
$q \in [1,\infty)$
satisfy
$2 \leq \alpha q < 2q+1$.


\subsection{Proof of Theorem \ref{orderinscribed}}\label{pforderinscribed}

Firstly,
we have
\begin{eqnarray*}
	\lefteqn{
	| \mathcal{E}^{\alpha ,q} (\gamma) - \mathcal{E}_n^{\alpha ,q} (p_n) |
	}
\\
	& \leq &
	\sum_{i=1}^n \sum_{|j-i| \leq N}
	\int_{b_j}^{b_{j+1}} \int_{b_i}^{b_{i+1}}
	\left|
	\Mag^q - \Manp^q
	\right|
	dsdt
\\
	& &
	+\sum_{i=1}^n \sum_{|j-i| >N}
	\int_{b_j}^{b_{j+1}} \int_{b_i}^{b_{i+1}}
	\left|
	\Mag^q - \Manp^q
	\right|
	dsdt,
\end{eqnarray*}
where
$ \displaystyle \sum_{|j-i| \leq N} $
and
$ \displaystyle \sum_{|j-i| > N} $
are summations with respect to $j$ with
$ |j-i| \leq N $ and $ |j-i| > N $ for each $i=1, \ldots ,n$ respectively.
In what follows, we estimate each of them.

\subsubsection{Estimates for the case where $ |j-i| \leq N $}

\begin{prop}\label{estleqN}

We have
\begin{align*}
	& \ L^{\alpha q - 2}
	\sum_{ i=1 }^n \sum_{ | j-i | \leq N }
	\int_{ b_j }^{ b_{ j+1 } } \int_{ b_i }^{ b_{ i+1 } }
	\Mag^q
	dsdt
\\
	\leq & \ 
	\frac{ 
	\left( \alpha  +2 \right)^q C_b^{ \alpha q/2 } 
	\{ \bar{c} (N+1) \}^{2q- \alpha q+2} ( 2N + 1 )
	}
	{ 4^q (2q- \alpha q+1)(2q- \alpha q+2)}
	L^{ 2q } K^{ 2q }
	\frac{1}{ n^{2q- \alpha q+1} }.
\end{align*}
Moreover, if $\alpha \leq 2$, then we have
\[
	L^{\alpha q - 2}
	\sum_{ i=1 }^n \sum_{ | j-i | \leq N }
	\int_{ b_j }^{ b_{ j+1 } } \int_{ b_i }^{ b_{ i+1 } }
	\Mag^q
	dsdt
	\leq
	\frac{C_b^{\alpha q+1} (2N+1) \bar{c}}{2^{\alpha q/2}}
	L^{\alpha q} K^{\alpha q}
	\frac 1 n.
\]

\end{prop}

\begin{proof}
We have
\begin{eqnarray*}
	\lefteqn{
	L^{\alpha q - 2}
	\sum_{ i=1 }^n \sum_{ | j-i | \leq N }
	\int_{ b_j }^{ b_{ j+1 } } \int_{ b_i }^{ b_{ i+1 } }
	\Mag^q
	dsdt} 
\\
	& \overset{(\ref{bi-Lip})}{\leq} & 
	L^{\alpha q - 2}
	C_b^{ \alpha q }
	\sum_{ i=1 }^n \sum_{ | j-i | \leq N }
	\int_{ b_j }^{ b_{ j+1 } } \int_{ b_i }^{ b_{ i+1 } }
	\frac{ ( | t - s |^\alpha - | \gamma (t) - \gamma (s) |^\alpha )^q }
	{ | t-s |^{ 2 \alpha q } }
	dsdt 
\\
	& \overset{(\ref{alpha2more--est.})}{\leq} &
	L^{\alpha q - 2}
	C_b^{ \alpha q }
	\left( \frac \alpha 2 +1 \right)^q
	\frac{ K^{ 2q } }{ 2^q }
	\sum_{ i=1 }^n \sum_{ | j-i | \leq N }
	\int_{ b_j }^{ b_{ j+1 } } \int_{ b_i }^{ b_{ i+1 } }
	| t-s |^{ 2q- \alpha q }
	dsdt 
\\
	& \overset{(\ref{div.bk})}{\leq} &
	\frac{ 
	\left( \alpha  +2 \right)^q C_b^{ \alpha q } 
	\{ \bar{c} (N+1) \}^{2q- \alpha q+2} ( 2N + 1 )
	}
	{ 4^q (2q- \alpha q+1)(2q- \alpha q+2)}
	L^{ 2q } K^{ 2q }
	\frac{1}{ n^{2q- \alpha q+1} }.
\end{eqnarray*}

In the case where
$ \alpha \leq 2 $,
we have
\begin{eqnarray*}
	\lefteqn{
	L^{\alpha q - 2}
	\sum_{ i=1 }^n \sum_{ | j-i | \leq N }
	\int_{ b_j }^{ b_{ j+1 } } \int_{ b_i }^{ b_{ i+1 } }
	\Mag^q
	dsdt
	}
\\
	& \overset{(\ref{bi-Lip})}{\leq} & 
	L^{\alpha q - 2}
	C_b^{ \alpha q }
	\sum_{ i=1 }^n \sum_{ | j-i | \leq N }
	\int_{ b_j }^{ b_{ j+1 } } \int_{ b_i }^{ b_{ i+1 } }
	\frac{ ( | t - s |^\alpha - | \gamma (t) - \gamma (s) |^\alpha )^q }
	{ | t-s |^{ 2 \alpha q } }
	dsdt
\\
	& \overset{(\ref{alpha2less--est.})}{\leq} &
	L^{\alpha q - 2}
	C_b^{ \alpha q } \frac{ K^{ \alpha q } }{ 2^{ \alpha q /2}}
	\sum_{ i=1 }^n \sum_{ | j-i | \leq N }
	| b_{j+1} - b_j | | b_{ i+1 } - b_i |
\\
	& \overset{(\ref{div.bk})}{\leq} &
	\frac{ C_b^{ \alpha q +1 } (2N+1) \bar{c} }{ 2^{ \alpha q/2 }}
	L^{ \alpha q } K^{ \alpha q }
	\frac{ 1 }{ n }.
\end{eqnarray*}
\end{proof}

The following proposition is proved by the same calculations as those in the proof of Proposition \ref{estleqN}.

\begin{prop}\label{estleqN2}
We have
\begin{align*}
	&\ 
	\tilde{L}_n^{\alpha q - 2}
	\sum_{ i=1 }^n \sum_{ | j-i | \leq N }
	\Manp^q
	| \gamma ( b_{ i+1 } ) - \gamma ( b_i ) |
	| \gamma ( b_{ j+1 } ) - \gamma ( b_j ) |
\\
	\leq &\ 
	\frac{
	( \alpha +2 )^q C_b^{ \alpha q } 
	c^{2q- \alpha q} \bar{c}^2 ( 2N+1 )
	}
	{ 4^q }
	L^{2q} K^{2q}
	\frac{1}{ n^{2q- \alpha q+1} }.
\end{align*}
Moreover, if $\alpha \leq 2 $, then we have
\begin{align*}
	&\ \tilde{L}_n^{\alpha q - 2}
	\sum_{ i=1 }^n \sum_{ | j-i | \leq N }
	\Manp^q
	| \gamma ( b_{ i+1 } ) - \gamma ( b_i ) |
	| \gamma ( b_{ j+1 } ) - \gamma ( b_j ) |
\\
	\leq &\ 
	\frac{ C_b^{ \alpha q } (2N+1) \bar{c} }{ 2^{ \alpha q /2} }
	L^{ \alpha q } K^{ \alpha q }
	\frac{1}{n}.
\end{align*}
\end{prop}

\subsubsection{Estimates for the case where $ |j-i| > N $}

Note that the difference between the lengths of the curve and its inscribed polygon satisfies
\begin{equation}
	L - \tilde{L}_n
	\leq
	\frac{ \bar{c}^3 C_b^3}{ 2 }
	L^3 K^2
	\frac{ 1 }{ n^2 },
	\label{L-L--est.}
\end{equation}
which follows from (\ref{div.bk}), (\ref{bi-Lip}), and (\ref{ts--est.}).
In order to determine how to prove Theorem \ref{orderinscribed},
we use the following lemma which may be proved by the same calculations as those in the proof of Proposition \ref{estleqN2}.

\begin{lem}\label{line}
We have
\begin{align*}
	&\ 
	\left|
	\sum_{ i=1 }^n \sum_{ | j-i | > N }
	\Manp^q
	| \gamma ( b_{ i+1 } ) - \gamma ( b_i ) |
	| \gamma ( b_{ j+1 } ) - \gamma ( b_j ) |
	\right|
\\
	\leq &\ 
	\frac{ C_b^{\alpha q} ( \alpha  +2 )^q c^{2q-\alpha q} \bar{c}^2 }
	{ 4^q (2q- \alpha q+1) }
	L^{2q- \alpha q+2} K^{2q}.
\end{align*}
Moreover, if $\alpha \leq 2$, we have
\[
	\left|
	\sum_{ i=1 }^n \sum_{ | j-i | > N }
	\Manp^q
	| \gamma ( b_{ i+1 } ) - \gamma ( b_i ) |
	| \gamma ( b_{ j+1 } ) - \gamma ( b_j ) |
	\right|
	\leq 
	\frac{ C_b^{\alpha q} \bar{c}^2 }{ 2^{\alpha q /2} }
	L^2 K^{\alpha q}.
\]
\end{lem}

Set
\begin{align*}
	X
	& := \,
	\sum_{ i=1 }^n \sum_{ | j-i | > N }
	\int_{ b_j }^{ b_{ j+1 } } \int_{ b_i }^{ b_{ i+1 } }
	\Mag^q
	dsdt,
\\
	Y
	& := \,
	\sum_{ i=1 }^n \sum_{ | j-i | > N }
	\Manp^q
	| \gamma ( b_{ i+1 } ) - \gamma ( b_i ) |
	| \gamma ( b_{ j+1 } ) - \gamma ( b_j ) |.
\end{align*}
We will estimate
$ | L^{ \alpha q -2 } X - \tilde{L}_n^{ \alpha q -2 } Y | $
which is the difference
of
the part of summation corresponding to
$i = 1, \ldots ,n$ and $ j $ with $ |j-i| > N $
of
$ \mathcal{E}^{\alpha,q} $
and
$ \mathcal{E}_n^{\alpha,q} $.
If there exists
$ \ell > 0 $
and
$ \tilde{C} > 0 $
such that
$ |X-Y| \leq \tilde{C} n^{- \ell} $,
then we have
\begin{eqnarray*}
	\lefteqn{
	| L^{ \alpha q -2 } X - \tilde{L}_n^{ \alpha q -2 } Y |
	}
\\
	& \overset{(\ref{1-xest.})}{\leq} &
	L^{ \alpha q -2 }
	| X-Y |
	+
	( \alpha q - 1 )
	L^{ \alpha q - 3 }
	( L - \tilde{L}_n ) | Y |
\\
	& \overset{\substack{(\ref{L-L--est.}),\\\text{Lemma }\ref{line}}}{\leq} &
	\tilde{C} L^{ \alpha q -2 } \frac{ 1 }{ n^{\ell} }
	+
	\frac{
	( \alpha q - 1 ) C_b^{\alpha q+3} ( \alpha  +2 )^q
	c^{2q-\alpha q} \bar{c}^5
	}
	{ 2^{2q+1} (2q- \alpha q+1) }
	L^{2q+2} K^{2q+2}
	\frac{ 1 }{ n^2 },
\end{eqnarray*}
and if $\alpha \leq 2$, similarly we have
\[
	| L^{ \alpha q -2 } X - \tilde{L}_n^{ \alpha q -2 } Y |
	\overset{\text{Lemma }\ref{line}}{\leq}
	\tilde{C} L^{ \alpha q -2 } \frac{ 1 }{ n^{\ell} }
	+
	\frac{
	( \alpha q - 1 ) C_b^{\alpha q+3} \bar{c}^5
	}
	{ 2^{\alpha q/2+1} }
	L^{ \alpha q +2 } K^{ \alpha q+2 }
	\frac{ 1 }{ n^2 }.
\]
Thus, it is sufficient to estimate
$ | X-Y | $.

Next, set
\fontsize{8.5pt}{0pt}\selectfont
\begin{align*}
	A_{i,j}
	& :=
	\int_{ b_j }^{ b_{ j+1 } } \!\! \int_{ b_i }^{ b_{ i+1 } }
	\left|
	\frac
	{ | t-s |^\alpha - | \gamma (t) - \gamma (s) |^\alpha }
	{ | \gamma (t) - \gamma (s) |^\alpha | t-s |^\alpha }
	-
	\frac
	{ | b_j - b_i |^\alpha - | \gamma ( b_j ) - \gamma ( b_i ) |^\alpha }
	{ | \gamma (t) - \gamma (s) |^\alpha | t-s |^\alpha }
	\right|
	dsdt,
\\
	B_{i,j}
	& :=
	\lefteqn{
	\int_{ b_j }^{ b_{ j+1 } } \!\! \int_{ b_i }^{ b_{ i+1 } }
	\left|
	\frac
	{ | b_j - b_i |^\alpha - | \gamma ( b_j ) - \gamma ( b_i ) |^\alpha }
	{ | \gamma (t) - \gamma (s) |^\alpha | t-s |^\alpha }
	-
	\frac
	{ | b_j - b_i |^\alpha - | \gamma ( b_j ) - \gamma ( b_i ) |^\alpha }
	{ | \gamma ( b_j ) - \gamma ( b_i ) |^\alpha | b_j - b_i |^\alpha }
	\right|
	dsdt,
	}
\\
	C_{i,j}
	& := 
	\left|
	\frac{ 1 }{ | a_j - a_i |^\alpha }
	-
	\frac{ 1 }{ | b_j - b_i |^\alpha }
	\right|
	| b_{i+1} - b_i |
	| b_{j+1} - b_j |,
\\
	D_{i,j}
	& :=
	\left|
	| b_{i+1} - b_i |
	| b_{j+1} - b_j |
	-
	| \gamma ( b_{i+1} ) - \gamma ( b_i ) |
	| \gamma ( b_{j+1} ) - \gamma ( b_j ) |
	\right|.
\end{align*}
\normalsize
\begin{rem}
In what follows,
$ C_{\rm g} $
is a positive constant that may change from line to line.
\end{rem}

Then, we have the following key lemma.
\begin{lem}\label{keylem}
There exists a positive constant
$C_{\rm g}$
such that
we have
\begin{align*}
	| X-Y |
	\leq &\ 
	C_{\rm g}
	K^{2( q-1 )}
	\sum_{ i=1 }^n \sum_{ | j-i | > N }
	| b_j - b_i |^{-( \alpha -2 )(q-1) }
	(A_{i,j} + B_{i,j} + C_{i,j} )
\\
	&\ 
	+
	C_{\rm g}
	K^{2q}
	\sum_{ i=1 }^n \sum_{ | j-i | > N }
	| b_j - b_i |^{-( \alpha -2)q} D_{i,j}.
\end{align*}
Moreover, if $\alpha \leq 2$,
we have
\[
	| X-Y |
	\leq
	C_{\rm g} K^{ \alpha ( q-1 ) }
	\sum_{ i=1 }^n \sum_{ | j-i | > N }
	(
	A_{i,j} + B_{i.j} + C_{i,j}
	)
	+
	C_{\rm g} K^{ \alpha q }
	\sum_{ i=1 }^n \sum_{ | j-i | > N } D_{i,j}.
\]
\end{lem}

\begin{proof}
We have
\begin{eqnarray}
	\lefteqn{
	\quad | X-Y |
	}
\nonumber\\
	& \overset{(\ref{1-xest.}),\,(\ref{alpha2more--est.})}{\leq} &
	C_{\text g} K^{2(q-1)}
	\sum_{ i=1 }^n \sum_{ | j-i | > N }
	\int_{ b_j }^{ b_{ j+1 } } \int_{ b_i }^{ b_{ i+1 } }
	\left|
	\mathcal{M}^\alpha ( \gamma )
	-
	\mathcal{M}_n^\alpha ( p_n )
	\right|
\nonumber\\
	& &
	\qquad \qquad \qquad
	\times
	\max
	\left\{
	\frac{ |t-s|^2 }{ |\gamma(t) - \gamma(s)|^\alpha}
	,
	\frac{ |b_j - b_i|^2 }{|\gamma(b_j) - \gamma(b_i)|^\alpha}
	\right\}^{ q-1 }
	dsdt
\nonumber\\
	& &
	+\, C_{\text g} K^{2q}
	\sum_{ i=1 }^n \sum_{ | j-i | > N }
	\frac{ |b_j - b_i|^{2q} }{ |\gamma(b_j) - \gamma(b_i)|^{\alpha q}}
\nonumber\\
	& &
	\qquad 
	\times
	\left|
	| b_{i+1} - b_i |
	| b_{j+1} - b_j |
	-
	| \gamma ( b_{i+1} ) - \gamma ( b_i ) |
	| \gamma ( b_{j+1} ) - \gamma ( b_j ) |
	\right|
\nonumber\\
	& \overset{(\ref{bi-Lip}),\,(\ref{tsgeqbjbi})}{\leq} &
	C_{\text g}
	K^{2(q-1)}
	\sum_{ i=1 }^n \sum_{ | j-i | > N }
	| b_j - b_i |^{-( \alpha -2 )(q-1) }
	( A_{i,j} + B_{i,j} + C_{i,j})
\nonumber\\
	& &
	+\,
	C_{\text g}
	K^{2q}
	\sum_{ i=1 }^n \sum_{ | j-i | > N }
	| b_j - b_i |^{-( \alpha -2)q} D_{i,j}.
\nonumber
\end{eqnarray}
In the case where $\alpha \leq 2$, we get the claim in a similar way using (\ref{alpha2less--est.}) instead of (\ref{alpha2more--est.}).

\end{proof}

Before we estimate the summations appearing in the statement of Lemma \ref{keylem},
we state inequalities used later.
The following lemma is proved by using inequalities (\ref{bi-Lip}), (\ref{tsgeqbjbi}), and (\ref{1-xest.}).

\begin{lem}\label{A--pre.est.12}
For
$ s \in [ b_i , b_{i+1} ],\ t \in [ b_j , b_{j+1} ] $,
we have
\begin{equation}
	| 
	 | \gamma ( b_j ) - \gamma ( b_i ) |^\alpha
	 -
	 | \gamma (t) - \gamma (s) |^\alpha
	|
	\leq
	C_{ \rm g }
	| b_j - b_i |^{ \alpha - 1 }
	\max_{ k =  1 , \ldots , n }
	| b_{ k+1 } - b_k |,
	\label{A--pre.est.1}
\end{equation}
and
\begin{equation}
	|
	 | b_j - b_i |^\alpha
	 -
	 | t-s |^\alpha
	|
	\leq
	C_{ \rm g }
	| b_j - b_i |^{ \alpha - 1 }
	\max_{ k =  1 , \ldots , n }
	| b_{ k+1 } - b_k |.
	\label{A--pre.est.2}
\end{equation}
\end{lem}

Using Lemma \ref{A--pre.est.12}, we estimate the summations appearing in Lemma \ref{keylem}.

\begin{prop}\label{A--est.}
We have
\[
	\sum_{ i=1 }^n \sum_{ | j-i | > N }
	| b_j-b_i |^{ -( \alpha -2)(q-1) }
	A_{i,j}
	\leq
	C_{ \mathrm{g} }
	L^{ 2q - \alpha q+2 } K^2
	\frac{1}{ n^{ 2q- \alpha q+1} }.
\]
Moreover, if $\alpha \leq 2$, then we have
\[
	\sum_{ i=1 }^n \sum_{ | j-i | > N }
	A_{i,j}
	\leq
	C_{ \mathrm{g} }
	\left(
	L^2 K^\alpha
	\frac{ \log n }{ n }
	+
	L^{ 4-\alpha } K^2
	\frac{ 1 }{ n }
	\right).
\]
\end{prop}

\begin{proof}
Fix
$ s \in [ b_i , b_{i+1} ] $ and $ t \in [ b_j , b_{j+1} ] $.
Without loss of generality, we may assume $s<t$.
Then, we have
\begin{eqnarray}
	\lefteqn{
	| | t-s |^\alpha - | b_j - b_i |^\alpha |
	\left(
	1 -
	\frac
	{ | \gamma (t) - \gamma (s) |^\alpha }
	{ | t-s |^\alpha }
	\right)
	}
\nonumber\\
	& \overset{(\ref{A--pre.est.2}),\,(\ref{tsgeqbjbi})}{\leq} &
	C_{ \text g}
	\frac{ \max_{ k = 1 , \dots , n } | b_{ k+1 } - b_k | }
	{ | b_j - b_i | }
	( | t-s |^\alpha - | \gamma (t) - \gamma (s) |^\alpha )
\nonumber\\
	& \overset{(\ref{alpha2more--est.}),\,(\ref{tsleqbjbi})}{\leq} &
	C_{ \text g}
	K^2
	| b_j - b_i |^{ \alpha +1 }
	\max_{ k = 1 , \ldots , n } | b_{ k+1 } - b_k |.
	\label{A--pre.est.3}
\end{eqnarray}
Also, we have
\begin{eqnarray*}
	\lefteqn{
	|
	 ( | t-s |^2 - | \gamma (t) - \gamma (s) |^2 )
	 -
	 ( | b_j - b_i |^2 - | \gamma ( b_j ) - \gamma ( b_i ) |^2 )
	|
	}
\\
	& = &
	\left|
	\iint_{A}
	( 1- \langle \gamma^\prime (v) , \gamma^\prime (u) \rangle ) dudv
	-
	\iint_{B}
	( 1- \langle \gamma^\prime (v) , \gamma^\prime (u) \rangle ) dudv
	\right|,
\end{eqnarray*}
where
\begin{align*}
	A &:=\,
	( [s,t] \times [b_j,t] ) \cup ( [b_j,t] \times [s,t] ),
\\
	B &:=\,
	( [b_i,b_j] \times [b_i,s] ) \cup ( [b_i,s] \times [b_i,b_j] ).
\end{align*}
The integral over
$ [ b_j , t ] \times [ s,t ] $
is estimated as
\[
	\left|
	\int_s^t \int_{b_j}^t
	( 1- \langle \gamma^\prime (v) , \gamma^\prime (u) \rangle ) dudv
	\right|
	\leq
	\frac{K^2}{2}
	| t-s |^3 | t-b_j |.
\]
We can dominate the integrals over
$ [ s,t ] \times [ b_j ,t ] $
and
$ B $ similarly.
Then, we get
\begin{align}
	&\ 
	|
	 ( | t-s |^2 - | \gamma (t) - \gamma (s) |^2 )
	 -
	 ( | b_j - b_i |^2 - | \gamma ( b_j ) - \gamma ( b_i ) |^2 )
	|
\nonumber\\
	\leq &\ 
	2 K^2
	\left(
	| t-s |^3 | t - b_j |
	+
	| b_j - b_i |^3 | s - b_i |
	\right).
\label{A--pre.est.0}
\end{align}
Consequently, it holds that
\begin{eqnarray}
	\lefteqn{
	| b_j - b_i |^\alpha
	\left|
	\frac
	{ | \gamma (t) - \gamma (s) |^\alpha }
	{ | t-s |^\alpha }
	-
	\frac
	{ | \gamma ( b_j ) - \gamma ( b_i ) |^\alpha }
	{ | b_j - b_i |^\alpha }
	\right|
	}
\nonumber\\
	& \overset{\substack{(\ref{1-xest.}),\,(\ref{bi-Lip}),\\(\ref{alpha2less--est.}),\,(\ref{A--pre.est.0})}}{\leq} &
	C_{ \text g }
	K^2
	| b_j - b_i |^{ \alpha - 2 }
	\frac{
	(
	| t-s |^3 | t - b_j |
	+
	| b_j - b_i |^3 | s - b_i |
	)}
	{ |t-s|^2 }
\nonumber\\
	&  &
	+ \ C_{ \text g }
	K^2
	| b_j - b_i |^{ \alpha + 4 }
	\left|
	\frac{ 1 }{ | t-s |^2 }
	-
	\frac{ 1 }{ | b_j - b_i |^2 }
	\right|
\nonumber\\
	& \overset{(\ref{tsgeqbjbi}),\,(\ref{A--pre.est.2})}{\leq} &
	C_{ \text g }
	K^2
	| b_j - b_i |^{ \alpha + 1 }
	\max_{ k = 1 , \ldots , n } | b_{ k+1 } - b_k |.
	\label{A--pre.est.4}
\end{eqnarray}
Using (\ref{A--pre.est.3}) and (\ref{A--pre.est.4}), we have
\begin{eqnarray*}
	\lefteqn{
	\ |
	 (
	  | t-s |^\alpha
	  -
	  | \gamma (t) - \gamma (s) |^\alpha
	 )
	 -
	 (
	  | b_j - b_i |^\alpha
	  -
	  | \gamma ( b_j ) - \gamma ( b_i ) |^\alpha
	 )
	|}
\\
	& = &
	\left|
	( | t-s |^\alpha - | b_j-b_i |^\alpha )
	\left( 1- \frac{ | \gamma(t)-\gamma(s) |^\alpha }{ | t-s |^\alpha } \right)
	\right.
\\
	& &
	\left.
	-| b_j-b_i |^\alpha
	\left(
	\frac{ | \gamma(t)-\gamma(s) |^\alpha }{ | t-s |^\alpha }
	-
	\frac{ | \gamma(b_j)-\gamma(b_i) |^\alpha }{ | b_j-b_i |^\alpha }
	\right)
	\right|
\\
	& \leq &
	C_{ \text g }
	K^2
	| b_j - b_i |^{ \alpha + 1 }
	\max_{ k = 1 , \ldots , n } | b_{ k+1 } - b_k |.
\end{eqnarray*}
Hence, we have
\[
	A_{i,j}
	\leq
	C_{ \text g }
	K^2
	| b_j - b_i |^{ 1 - \alpha }
	\max_{ k = 1 , \cdots , n } | b_{ k+1 } - b_k |^3 ,
\]
and therefore, we get
\begin{eqnarray*}
	\lefteqn{
	\sum_{ i=1 }^n \sum_{ | j-i | > N }
	| b_j-b_i |^{ -( \alpha -2)(q-1) }
	A_{i,j}
	}
\\
	& \leq &
	C_{\text g}
	K^2
	\sum_{ i=1 }^n \sum_{ | j-i | > N }
	| b_j-b_i |^{ -( \alpha -2)(q-1)+1 -\alpha }
	\max_{ k = 1 , \ldots , n } | b_{ k+1 } - b_k |^3 
\\
	& \overset{(\ref{div.bk}),\,(\ref{bi-Lip})}{\leq} &
	C_{ \text g }
	L^{ 2q - \alpha q+2 } K^2
	\frac{1}{ n^{ 2q- \alpha q+1} }.
\end{eqnarray*}

Next, assume that
$ \alpha \leq 2 $.
Using (\ref{tsleqbjbi}), (\ref{tsgeqbjbi}), (\ref{alpha2less--est.}), and (\ref{A--pre.est.2}), we have
\begin{eqnarray*}
	\lefteqn{
	| | t-s |^\alpha - | b_j - b_i |^\alpha |
	\left(
	1 -
	\frac
	{ | \gamma (t) - \gamma (s) |^\alpha }
	{ | t-s |^\alpha }
	\right)
	}
\nonumber\\
	& \leq &
	C_{ \text g}
	K^\alpha
	| b_j - b_i |^{ 2 \alpha -1 }
	\max_{ k = 1 , \ldots , n } | b_{ k+1 } - b_k |.
\end{eqnarray*}
Therefore, using in addition (\ref{tsleqbjbi}), (\ref{tsgeqbjbi}), (\ref{A--pre.est.1}), (\ref{A--pre.est.2}), it holds that
\begin{eqnarray*}
	\lefteqn{
	|
	 (
	  | t-s |^\alpha
	  -
	  | \gamma (t) - \gamma (s) |^\alpha
	 )
	 -
	 (
	  | b_j - b_i |^\alpha
	  -
	  | \gamma ( b_j ) - \gamma ( b_i ) |^\alpha
	 )
	|}
\\
	& = &
	\left|
	( | t-s |^\alpha - | b_j-b_i |^\alpha )
	\left( 1- \frac{ | \gamma(t)-\gamma(s) |^\alpha }{ | t-s |^\alpha } \right)
	\right.
\\
	& &
	\left.
	-| b_j-b_i |^\alpha
	\left(
	\frac{ | \gamma(t)-\gamma(s) |^\alpha }{ | t-s |^\alpha }
	-
	\frac{ | \gamma(b_j)-\gamma(b_i) |^\alpha }{ | b_j-b_i |^\alpha }
	\right)
	\right|
\\
	& \leq &
	C_{ \text g }
	K^\alpha
	| b_j - b_i |^{ 2 \alpha - 1 }
	\max_{ k = 1 , \ldots , n } | b_{ k+1 } - b_k |
\\
	& &
	+\,C_{\text g}
	K^2
	| b_j - b_i |^{ \alpha + 1 }
	\max_{ k = 1 , \ldots , n } | b_{ k+1 } - b_k |.
\end{eqnarray*}
Moreover, by (\ref{div.bk}) and (\ref{bi-Lip}), we get
\[
	A_{i,j}
	\leq
	C_{\text g}
	\left(
	L^2 K^\alpha \frac 1 { |j-i| n^2 } + L^{4- \alpha} K^2 \frac 1 { |j-i|^{\alpha -1} n^{4- \alpha } }
	\right).
\]
Hence, we obtain
\[
	\sum_{ i=1 }^n \sum_{ | j-i | > N }
	A_{i,j}
	\leq
	C_{ \text g }
	\left(
	L^2 K^\alpha
	\frac{ \log n }{ n }
	+
	L^{ 4-\alpha } K^2
	\frac{ 1 }{ n }
	\right).
\]
\end{proof}

\begin{prop}\label{B--est.}
We have
\[
	\sum_{ i=1 }^n \sum_{ | j-i | > N }
	| b_j-b_i |^{ -( \alpha -2 )(q-1) }
	B_{i,j}
	\leq 
	C_{ \mathrm{g} }
	L^{ 2q - \alpha q+2 } K^2
	\frac{1}{ n^{ 2q- \alpha q+1} }.
\]
Moreover, if $\alpha \leq 2$, then we have
\[
	\sum_{ i=1 }^n \sum_{ | j-i | > N }
	B_{i,j}
	\leq
	C_{ \mathrm{g} }
	L^2 K^\alpha
	\frac{ \log n }{ n }.
\]
\end{prop}

\begin{proof}
Since we have
\begin{align*}
	&\ 
	|
	 | \gamma ( b_j ) - \gamma ( b_i ) |^\alpha | b_j - b_i |^\alpha 
	 -
	 | \gamma (t) - \gamma (s) |^\alpha | t-s |^\alpha
	| 
\\
	\leq &\ 
	C_{ \text g }
	| b_j - b_i |^{ 2 \alpha - 1 }
	\max_{ k =  1 , \ldots , n }
	| b_{ k+1 } - b_k |
\end{align*}
using (\ref{A--pre.est.1}) and (\ref{A--pre.est.2}), we have
\begin{eqnarray}
	\lefteqn{
	\int_{ b_j }^{ b_{ j+1 } }
	\int_{ b_i }^{ b_{ i+1 } }
	\frac
	{ 
	|
	 | \gamma ( b_j ) - \gamma ( b_i ) |^\alpha | b_j - b_i |^\alpha 
	 -
	 | \gamma (t) - \gamma (s) |^\alpha | t-s |^\alpha
	|
	}
	{ 
	| \gamma (t) - \gamma (s) |^\alpha | t-s |^\alpha 
	| \gamma ( b_j ) - \gamma ( b_i ) |^\alpha | b_j - b_i |^\alpha
	}
	dsdt
	}
\nonumber\\
	& \leq &
	C_{\text g}
	\int_{ b_j }^{ b_{ j+1 } }
	\int_{ b_i }^{ b_{ i+1 } }
	\frac
	{
	| b_j - b_i |^{ \alpha - 1 }
	\max_{ k =  1 , \ldots , n }
	| b_{ k+1 } - b_k |
	}
	{ 
	| \gamma (t) - \gamma (s) |^\alpha | t-s |^\alpha 
	| \gamma ( b_j ) - \gamma ( b_i ) |^\alpha
	}
	dsdt
\nonumber\\
	& \overset{\substack{(\ref{div.bk}),\ (\ref{bi-Lip}),\\(\ref{tsleqbjbi}),\ (\ref{tsgeqbjbi})}}{\leq} &
	C_{\text g}
	\frac{ \max_{ k=1,\ldots,n } | b_{k+1} - b_k |^3 }
	{ | b_j - b_i |^{ 2 \alpha +1 } },
\nonumber
\end{eqnarray}
and therefore, using (\ref{alpha2more--est.}),
we have
\[
	B_{i,j}
	\leq
	C_{ \text g }
	K^2
	| b_j - b_i |^{1- \alpha }
	\max_{ k =  1 , \ldots , n }
	| b_{ k+1 } - b_k |^3.
\]
Hence, by (\ref{div.bk}) and (\ref{bi-Lip}), we get
\[
	\sum_{ i=1 }^n \sum_{ | j-i | > N }
	| b_j-b_i |^{ -( \alpha -2 )(q-1) }
	B_{i,j}
	\leq
	C_{ \text g }
	L^{ 2q - \alpha q +2 } K^2
	\frac{1}{ n^{ 2q- \alpha q+1} }.
\]

If $ \alpha \leq 2 $, using (\ref{alpha2less--est.}) instead of (\ref{alpha2more--est.}), similarly we have
\[
	B_{i,j}
	\leq
	C_{\text g} | b_j-b_i |^{-1}
	\max_{k=1, \ldots ,n} | b_{k+1}-b_k |.
\]
Therefore, by (\ref{div.bk}) and (\ref{bi-Lip}), we get
\[
	\sum_{ i=1 }^n \sum_{ | j-i | > N }
	B_{i,j}
	\leq
	C_{ \text g }
	L^2 K^\alpha
	\frac{ \log n }{ n }.
\]
\end{proof}

\begin{prop}\label{C--est.}
We have
\[
	\sum_{ i=1 }^n \sum_{ | j-i | > N }
	| b_j-b_i |^{ -( \alpha -2)(q-1) }
	C_{i,j}
	\leq
	C_{ \mathrm{g} }
	L^{2q- \alpha q+2} K^2
	\frac{1}{ n^{ 2q- \alpha q+1 } }.
\]
Moreover, if $ \alpha \leq 2 $, then we have
\[
	\sum_{i=1}^n \sum_{ |j-i| > N } C_{i,j}
	\leq
	\begin{cases}
		\displaystyle
		C_{\mathrm{g}} L^{4- \alpha } K^2 \frac{1}{n^2}
		& ( 0 < \alpha < 1 ),
	\vspace{2mm}\\
		\displaystyle
		C_{\mathrm{g}} L^3 K^2 \frac{\log n}{n^2}
		& ( \alpha =1 ),
	\vspace{2mm}\\
		\displaystyle
		C_{\mathrm{g}} L^{4- \alpha } K^2 \frac{1}{n^{3- \alpha}}
		& ( 1 < \alpha \leq 2 ).		
	\end{cases}
\]
\end{prop}

\begin{proof}
We may assume
$ j>i $ because of the symmetry of $i$ and $j$.
Also, since
\[
	| b_j - b_i |
	=
	\min \left\{
	\sum_{k=i}^{j-1} | b_{k+1} - b_k |
	,
	\sum_{k=j}^{i+n-1} | b_{k+1} - b_k |
	\right\},
\]
we may assume
\[
	| b_j - b_i |
	=
	\sum_{k=i}^{j-1} | b_{k+1} - b_k |.
\]
Otherwise, we reduce to the above case by changing
$ \{ j , i+n \} $
with
$ \{ i,j \} $.
In this situation, we have
\begin{eqnarray}
	\lefteqn{
	| b_j - b_i |^\alpha
	-
	| a_j - a_i |^\alpha
	}
\nonumber\\
	& \overset{(\ref{1-xest.})}{\leq} &
	\left( 
	{\frac \alpha 2} + 1 
	\right)
	| b_j - b_i |^{ \alpha - 2 }
	( | b_j - b_i | + | a_j - a_i | )
	| | b_j - b_i | - | a_j - a_i | |
\nonumber\\
	& \overset{(\ref{ts--int.est.})}{\leq} &
	2 \left( \frac \alpha 2 + 1 \right)
	K^2
	| b_j - b_i |^\alpha
	\max_{ k = 1 , \ldots , n }
	| b_{ k+1 } - b_k |^2.
\nonumber
\end{eqnarray}
Hence, we have
\[
	C_{i,j}
	\overset{(\ref{bi-Lip}),\,(\ref{ab--bi.Lip.})}{\leq}
	C_{ \text g }
	K^2
	| b_j-b_i |^{ -\alpha }
	\max_{ k = 1 , \ldots , n }
	| b_{ k+1 } - b_k |^4.
\]
Therefore,
we get
\[
	\sum_{ i=1 }^n \sum_{ | j-i | > N }
	| b_j-b_i |^{ -( \alpha -2)(q-1) }
	C_{i,j}
	\overset{(\ref{div.bk}),\,(\ref{bi-Lip})}{\leq}
	C_{ \text g }
	L^{2q- \alpha q+2} K^2
	\frac{1}{ n^{ 2q- \alpha q+1 } },
\]
and if $\alpha \leq 2$, since we have
\[	
	\sum_{i=1}^n \sum_{ |j-i| >N } C_{i,j}
	\overset{(\ref{div.bk}),\,(\ref{bi-Lip})}{\leq}
	C_{ \text g }
	L^{4- \alpha } K^2
	\frac{1}{ n^{3- \alpha} }
	\sum_{k=1}^n \frac{1}{k^\alpha},
\]
we get the claim by estimating $ \sum_{k=1}^n 1/k^\alpha $.
\end{proof}

\begin{prop}\label{D--est.}
We have
\[
	\sum_{i=1}^n \sum_{ |j-i| <N }
	| b_j - b_i |^{-(\alpha -2) q} D_{i,j}
	\leq
	C_{ \rm g }
	L^{2q - \alpha q + 4} K^2
	\frac{1}{n^2}.
\]
Moreover, if $\alpha \leq 2$, then we have
\[
	\sum_{i=1}^n \sum_{ |j-i| <N } D_{i,j}
	\leq
	C_{ \rm g } L^4 K^2 \frac{1}{n^2}.
\]
\end{prop}

\begin{proof}
Since we have
\[
	D_{i,j}
	\leq
	C_{ \text g }
	K^2
	\max_{ k = 1 , \ldots , n }| b_{ k+1 } - b_k |^4
\]
using (\ref{bi-Lip}) and (\ref{ts--est.}),
we get
\begin{eqnarray*}
	\lefteqn{
	\sum_{ i=1 }^n \sum_{ | j-i | > N }
	| b_j - b_i |^{-( \alpha -2)q}
	D_{i,j}
	}
\\
	& \leq &
	C_{ \text g }
	K^2
	\max_{ k = 1 , \ldots , n }| b_{ k+1 } - b_k |^4
	\sum_{ i=1 }^n \sum_{ | j-i | > N }
	| b_j - b_i |^{-( \alpha -2)q}
\\
	& \overset{(\ref{div.bk}),\,(\ref{bi-Lip})}{\leq} &
	C_{\text g}
	L^{ 2q- \alpha q +4 } K^2
	\frac{1}{n^2},
\end{eqnarray*}
and we get
\[
	\sum_{i=1}^n \sum_{ |j-i| <N } D_{i,j}
	\leq
	C_{ \text g }
	K^2
	\max_{ k = 1 , \ldots , n }| b_{ k+1 } - b_k |^4
	\sum_{i=1}^n \sum_{ |j-i| <N } 1
	\overset{(\ref{div.bk}),\ (\ref{bi-Lip})}{\leq}
	C_{ \rm g } L^4 K^2 \frac{1}{n^2}
\]
if $ \alpha \leq 2$.
\end{proof}

Using Propositions \ref{A--est.}--\ref{D--est.},
we get
\begin{align*}
	| X-Y |
	& \leq \,
	C_{\text g}
	K^{2( q-1 )}
	\sum_{ i=1 }^n \sum_{ | j-i | > N }
	| b_j - b_i |^{-( \alpha -2 )(q-1) }
	( A_{i,j} + B_{i,j} + C_{i,j} )
\\
	&\,
	\quad+
	C_{\text g}
	K^{2q}
	\sum_{ i=1 }^n \sum_{ | j-i | > N }
	| b_j - b_i |^{-( \alpha -2)q} D_{i,j}
\\
	& \leq \,
	C_{\text g}
	\left(
	L^{2q - \alpha q +2} K^{2q} + L^{2q - \alpha q +4} K^{2q+2} 
	\right)
	\frac{1}{n^{2q - \alpha q +1}}.
\end{align*}
Moreover, in the case where $ \alpha \leq 2 $,
we have
\begin{eqnarray*}
	\lefteqn{
	| X-Y |
	}
\\
	& \leq &
	C_{\text g} K^{ \alpha ( q-1 ) }
	\sum_{ i=1 }^n \sum_{ | j-i | > N }
	( A_{i,j} + B_{i,j} + C_{i,j} )
	+
	C_{\text g} K^{ \alpha q }
	\sum_{ i=1 }^n \sum_{ | j-i | > N } D_{i,j}
\\
	& \leq &
	\left\{
	\begin{array}{l}
		\displaystyle
		C_{\text g}
		\left\{
		L^2 K^{\alpha q} \frac{\log n}{n}
		+ L^{4-\alpha} K^{ \alpha q -\alpha +2 }
		\left( \frac 1 n + \frac{\log n}{n^2} \right)
		+ L^4 K^{\alpha q+2} \frac{1}{n^2}
		\right\}
	\vspace{1mm}\\
		\hspace{8.5cm}
		(0 < \alpha < 1),
	\vspace{2mm}\\
		\displaystyle
		C_{\text g}
		\left\{
		L^2 K^q \frac{\log n}{n}
		+ L^3 K^{q+1}
		\left( \frac 1 n + \frac{1}{n^2} \right)
		+ L^4 K^{q+2} \frac{1}{n^2}
		\right\}\ \ 
		(\alpha=1),
	\vspace{2mm}\\
		\displaystyle
		C_{\text g}
		\left\{
		L^2 K^{\alpha q} \frac{\log n}{n}
		+ L^{4-\alpha} K^{ \alpha q -\alpha +2 }
		\left( \frac 1 n + \frac{1}{n^{3-\alpha}} \right)
		+ L^4 K^{\alpha q+2} \frac{1}{n^2}
		\right\}
	\vspace{1mm}\\
		\hspace{8.5cm}
		(1 < \alpha \leq 2).
	\end{array}
	\right.
\end{eqnarray*}
Thus, we get
\[
	| X-Y |
	\leq
	C_{\text g}
	\left(
	L^2 K^{\alpha q} + L^{4- \alpha} K^{\alpha q - \alpha +2}
	+ L^4 K^{\alpha q+2}
	\right)
	\frac{\log n}{n}.
\]

This completes our proof of Theorem \ref{orderinscribed}.
\qquad \qquad \qquad \qquad \qquad \qquad \quad \ \ \,$\square$


\subsection{Proof of Theorem \ref{inscribed}}\label{pfinscribed}

Set
\[
	\varepsilon_n
	:=
	L^{ \alpha q - 2 } 
	\sum_{ k=1 }^n
	\int_{ b_k }^{ b_{ k+1 } }
	\int_{ b_k }^{ b_{ k+1 } }
	\frac
	{ | \gamma^\prime (v) - \gamma^\prime (v) |^{ 2q } }
	{ | v-u |^{ \alpha q } }
	dudv
\]
for
$ \gamma \in W^{1+\sigma , 2q} ( \mathbb{S}_L , \mathbb{R}^d )$.
Note that
$ \varepsilon_n < \infty $
because
$ \alpha q = 1+2 \sigma q $.
Since
\[
	\mu
	\left(
	\bigcup_{ k=1 }^n
	[ b_k , b_{k+1} ]
	\times
	[ b_k , b_{k+1} ]
	\right)
	\to	0
\]
as $ n \to \infty $,
we have
$ \varepsilon_n \to 0 $
from the absolute continuity of integrals for absolutely integrable functions,
where
$\mu$
is the Lebesgue measure on
$\mathbb{S}_L \times \mathbb{S}_L$.

Using
$ \varepsilon_n $,
set
$
	N_n := n \max 
	\left\{
	\varepsilon_n^{ \frac{ 1 }{ 4 q } } 
	,
	n^{ - \frac{ 1 }{ 6 q } }
	\right\}.
$

\subsubsection{Estimates for the case where $ |j-i| \leq N_n $}

Set
$
	\displaystyle \delta
	= \left( 1 + 2 \frac{ \bar{c} }{ c } \,\right)^{-1}
$.
Let
$ n \in \mathbb{N} $
be sufficiently large
such that $ \{ b_k \} $ satisfies
\[
	| b_{ k+1 } - b_k |
	\leq
	( 1 + \delta )
	| \gamma ( b_{ k+1 } ) - \gamma ( b_k ) |.
\]
Then, since we have
\begin{align*}
	| b_{ k+1 } - b_k |
	& \leq \,
	( 1- \delta +2 \delta )
	| \gamma ( b_{ k+1 } ) - \gamma ( b_k ) |
\\
	& \leq \,
	( 1 - \delta )
	( | b_{ k+1 } - b_k | + | b_k - b_{k-1} | )
\\
	& = \,
	( 1 - \delta ) | b_{ k+1 } - b_{k-1} |
\end{align*}
using (\ref{div.bk}), we have
\begin{equation}
	\delta
	| b_{ j+1 } - b_i |
	\leq 
	| b_{ j+1 } - b_i |
	-
	( 1- \delta )
	| b_{ j+1 } - b_{j-1} |
	\leq
	| b_j - b_i |.
	\label{delta}
\end{equation}
Therefore, we get
\small
\begin{eqnarray*}
	\lefteqn{
	\quad
	\tilde{L}_n^{\alpha q - 2}
	\sum_{ i=1 }^n \sum_{ | j-i | \leq N_n }
	\Manp^q
	| \gamma ( b_{ i+1 } ) - \gamma ( b_i ) |
	| \gamma ( b_{ j+1 } ) - \gamma ( b_j ) |
	}
\\
	& \overset{\substack{(\ref{ab--bi.Lip.}),\,(\ref{div.bk}),\\(\ref{bi-Lip}),\,(\ref{1-xest.})}}{\leq} &
	\tilde{L}_n^{\alpha q - 2}
	C_b^{ \alpha q }
	\frac{ \bar{c} }{ c }
	\left(
	{\frac \alpha 2}
	+ 1
	\right)^q
\\
	& & 
	\times
	\sum_{ i=1 }^n \sum_{ | j-i | \leq N_n }
	\left(
	\frac
	{
	| b_j - b_i |^2
	-
	| \gamma (b_j) - \gamma (b_i) |^2
	}
	{ | b_j - b_i |^{ \alpha + 2 } }
	\right)^q
	| b_i - b_{ i-1 } |
	| b_{ j+1 } - b_j |
\\
	& \overset{(\ref{delta}),\,(\ref{tsleqbjbi})}{\leq} &
	C_{\text g} \tilde{L}_n^{\alpha q - 2}
	\delta^{ -2 q ( \alpha +2 )}
\\
	& & 
	\times
	\sum_{ i=1 }^n \sum_{ | j-i | \leq N_n }
	\int_{ b_j }^{ b_{ j+1 } } \int_{ b_{ i-1 } }^{ b_i } 
	\left(
	\frac{
	\int_s^t \int_s^t 
	| \gamma^\prime ( v ) - \gamma^\prime ( u ) |^2
	dudv}
	{ 2 | t-s |^{ \alpha + 2 } }
	\right)^q
	dsdt
\\
	& \overset{(\ref{1-xest.})}{\leq} &
	C_{\text g}
	L^{\alpha q - 2}
	\delta^{ -2 q ( \alpha +2 )}
	\sum_{ i=1 }^n \sum_{ | j-i | \leq N_n }
	\int_{ b_j }^{ b_{ j+1 } } \int_{ b_{ i-1 } }^{ b_i } 
	\Mag^q
	dsdt
\\
	& \longrightarrow &
	0
\end{eqnarray*}
\normalsize
as $ n \to \infty $.
Here, we have used
\[
	\mu
	\left(
	\bigcup_{ | j-i | \leq N_n }
	[ b_{ i-1 } , b_i ]
	\times
	[ b_j , b_{ j+1 } ]
	\right)
	\leq
	2 C_b^2 \bar{c}^2
	\frac{ 1 }{ n }
	( N_n + 1 )
	\to
	0
\]
and the absolute continuity of the integral.
Also, we have
\[
	L^{\alpha p - 2}
	\sum_{ i=1 }^n \sum_{ | j-i | \leq N }
	\int_{ b_j }^{ b_{ j+1 } } \int_{ b_i }^{ b_{ i+1 } }
	\Mag^q
	dsdt
	\to
	0,
\]
which follows easily from the absolute continuity of the integral.

\subsubsection{Estimates for the case where $ |j-i| > N_n $}

First, by estimates (\ref{div.bk}) and (\ref{bi-Lip}), we have
\begin{equation}
	| \gamma ( b_j ) - \gamma ( b_i ) |
	\geq
	C_b c L \frac{ N_n }{ n }.
	\label{>Nn--est.}
\end{equation}

The strategy of the proof of Theorem \ref{inscribed} is as follows.
Note that
\begin{eqnarray*}
	\lefteqn{
	L - \tilde{L}_n
	}
\\
	& = &
	\sum_{ k=1 }^n
	(
	| b_{ k+1 } - b_k |
	-
	| \gamma ( b_{ k+1 } ) - \gamma ( b_k ) |
	)
\\
	& \overset{(\ref{ts--int.est.})}{\leq} &
	\frac{1}{2}
	\sum_{ k=1 }^n
	| b_{ k+1 } - b_k |^{ \alpha + 1 - 2/q }
	\left(
	\int_{ b_k }^{ b_{ k+1 } } \int_{ b_k }^{ b_{ k+1 } }
	\frac
	{ | \gamma ' (v) - \gamma ' (u) |^{ 2q } }
	{ | v-u |^{ \alpha q } }  dudv
	\right)^{ 1/q }
\\
	& \overset{(\ref{div.bk})}{\leq} &
	C_{\text g}
	L
	\frac{ 1 }{ n^{ \alpha - 1/q } }
	\varepsilon_n^{ 1/q }
\end{eqnarray*}
by H\"{o}lder's inequality, and
\begin{eqnarray*}
	|Y|
	& \leq &
	\sum_{ i=1 }^n  \sum_{ | j-i | > N_n }
	\frac{ 1 }{ | \gamma ( b_j ) - \gamma ( b_i ) |^{ \alpha q } }
	| \gamma ( b_{ i+1 } ) - \gamma ( b_i ) |
	| \gamma ( b_{ j+1 } ) - \gamma ( b_j ) |
\\
	& \overset{(\ref{>Nn--est.})}{\leq} &
	C_{ \text g } L^{ - \alpha q }
	\left(
	\frac{ n }{ N_n }
	\right)^{ \alpha q }
	\sum_{ i=1 }^n  \sum_{ | j-i | > N_n }
	| \gamma ( b_{ i+1 } ) - \gamma ( b_i ) |
	| \gamma ( b_{ j+1 } ) - \gamma ( b_j ) |
\\
	& \overset{(\ref{div.bk})}{\leq} &
	C_{ \text g }
	L^{ 2 - \alpha q } n^{ \alpha / 6 }.
\end{eqnarray*}
Then, if
$ | X-Y | \to 0 $
as $ n \to \infty $,
we get
\begin{align*}
	| L^{ \alpha q -2 } X - \tilde{L}_n^{ \alpha q -2 } Y |
	\,\leq \,&\ 
	L^{ \alpha q -2 }
	| X-Y |
	+
	( \alpha q - 1 )
	L^{ \alpha q - 3 }
	( L - \tilde{L}_n ) | Y |
\\
	\,\leq \,&\ 
	C_{ \text g }
	\left(
	L^{ \alpha q -2 }
	| X-Y |
	+
	\frac{ 1 }{ n^{ \frac{ 5 }{ 6 } \alpha - \frac{ 1 }{ q } } }
	\varepsilon_n^{ 1/q }
	\right)
\\	
	\,\,\to &\ 
	0,
\end{align*}
and Theorem \ref{inscribed} will be proved.
Here, we  have used
$ \displaystyle \frac 5 6 \alpha - \frac 1 q > 0$.

Thus, it suffices to prove
\[
	| X-Y | \to 0
\]
as $ n \to \infty $.
To this end,
observe that we have
\begin{eqnarray*}
	\lefteqn{| X-Y |}
\\
	& \overset{(\ref{1-xest.})}{\leq} &
	( q+1 )
	\sum_{ i=1 }^n \sum_{ | j-i | > N }
	\int_{ b_j }^{ b_{ j+1 } } \int_{ b_i }^{ b_{ i+1 } }
	\max
	\{
	| \mathcal{M}^\alpha ( \gamma ) |
	,
	| \mathcal{M}_n^\alpha ( p_n ) |
	\}^{ q-1 }
\nonumber\\
	& &
	\qquad \qquad \qquad \qquad \qquad \qquad \qquad \qquad
	\times\left|
	\mathcal{M}^\alpha ( \gamma )
	-
	\mathcal{M}_n^\alpha ( p_n )
	\right|
	dsdt
\nonumber\\
	& &
	+
	\sum_{ i=1 }^n \sum_{ | j-i | > N }
	\left|
	\mathcal{M}_n^\alpha ( p_n )
	\right|^q
\nonumber\\
	& &
	\qquad \quad
	\times
	\left|
	| b_{i+1} - b_i |
	| b_{j+1} - b_j |
	-
	| \gamma ( b_{i+1} ) - \gamma ( b_i ) |
	| \gamma ( b_{j+1} ) - \gamma ( b_j ) |
	\right|
\\
	& \leq &
	C_{ \text g }
	\sum_{ i=1 }^n \sum_{ | j-i | > N }
	\frac{ 1 }{ | \gamma ( b_j ) - \gamma ( b_i ) |^{ \alpha ( q-1 ) } }
	(
	A_{ i,j }
	+
	B_{ i,j }
	+
	C_{ i,j }
	)
\\
	& &
	+\,C_{\text g}
	\sum_{ i=1 }^n \sum_{ | j-i | > N }
	\frac{ 1 }{ | \gamma ( b_j ) - \gamma ( b_i ) |^{ \alpha q } }
	D_{i,j}
\\
	& \overset{(\ref{>Nn--est.})}{\leq} &
	C_{ \text g }
	L^{ - \alpha ( q-1 ) }
	\sum_{ i=1 }^n \sum_{ | j-i | > N }
	\max\left\{ 
	\varepsilon_n^{ \frac{ 1 }{ 4q } }
	,
	n^{ - \frac{ 1 }{ 6q } }
	\right\}^{ - \alpha ( q-1 ) } 
	(
	A_{ i,j }
	+
	B_{ i,j }
	+
	C_{ i,j }
	)
\\
	& &
	+\,C_{\text g}
	L^{ - \alpha q }
	\sum_{ i=1 }^n \sum_{ | j-i | > N }
	\max\left\{ 
	\varepsilon_n^{ \frac{ 1 }{ 4q } }
	,
	n^{ - \frac{ 1 }{ 6q } }
	\right\}^{ - \alpha q }
	D_{i,j}.
\end{eqnarray*}
We estimate these summations.

\begin{prop}\label{AB--est.2}
We have
\[
	\displaystyle
	\sum_{ i=1 }^n \sum_{ | j-i | > N }
	\max\left\{ 
	\varepsilon_n^{ \frac{ 1 }{ 4q } }
	,
	n^{ - \frac{ 1 }{ 6q } }
	\right\}^{ - \alpha ( q-1 ) } 
	( A_{i,j} + B_{i,j} )
	\leq
	C_{\rm g}
	L^{2- \alpha }
	\frac{ 1 }
	{ 
	n^{ 1- (\alpha q +1) / (6q) }
	}.
\]
\end{prop}

\begin{proof}
Note that
\begin{eqnarray}
	\lefteqn{
	|
	 ( | t-s |^2 - | \gamma (t) - \gamma (s) |^2 )
	 -
	 ( | b_j - b_i |^2 - | \gamma ( b_j ) - \gamma ( b_i ) |^2 )
	|}
\nonumber\\
	& = &
	\left|
	\int_A
	( 1- \langle \gamma^\prime (v) , \gamma^\prime (u) \rangle)
	dudv
	-
	\int_B
	( 1- \langle \gamma^\prime (v) , \gamma^\prime (u) \rangle)
	dudv
	\right|
\nonumber\\
	& \leq &
	C_{\text g}
	| t-s | | t -b_j | + | b_j-b_i | | s- b_i |,
	\label{pre--AB--est.}
\end{eqnarray}
where
\begin{align*}
	A &:=\,
	( [s,t] \times [b_j,t] ) \cup ( [b_j,t] \times [s,t] ),
\\
	B &:=\,
	( [b_i,b_j] \times [b_i,s] ) \cup ( [b_i,s] \times [b_i,b_j] ).
\end{align*}
Now, we have
\begin{eqnarray*}
	\lefteqn{
	A_{i,j}
	\overset{\substack{(\ref{1-xest.}),\,(\ref{bi-Lip}),\\(\ref{pre--AB--est.}),\,(\ref{A--pre.est.2})}}{\leq}
	C_{\text g}
	\int_{ b_j }^{ b_{j+1} } \int_{ b_i }^{ b_{i+1} }
	\frac{
	| b_j-b_i |^{\alpha -1} \max_{k=1, \ldots ,n} | b_{k+1} - b_k |
	}{ |t-s|^\alpha }
	dsdt
	}
\\
	& &
	+\,C_{\text g}
	\int_{ b_j }^{ b_{j+1} } \int_{ b_i }^{ b_{i+1} }
	\frac{
	| b_j-b_i |^\alpha \max_{k=1, \ldots ,n} | b_{k+1} - b_k |
	}{ |t-s|^{2 \alpha +1} }
	dsdt
\\
	& &
	+\,C_{\text g}
	\int_{ b_j }^{ b_{j+1} } \int_{ b_i }^{ b_{i+1} }
	\frac{
	| b_j-b_i |^{\alpha +1} \max_{k=1, \ldots ,n} | b_{k+1} - b_k |
	}{ |t-s|^{2 \alpha +2} }
	dsdt
\\
	& \overset{(\ref{tsleqbjbi}),\,(\ref{tsgeqbjbi})}{\leq} &
	C_{\text g}
	\frac{ \max_{k=1, \ldots ,n} | b_{k+1}-b_k |^3 }{ | b_j-b_i |^{\alpha +1} }.
\end{eqnarray*}
Also, we have
\[
	B_{i,j}
	\leq
	C_{\text g}
	\frac{ \max_{k=1, \ldots ,n} | b_{k+1}-b_k |^3 }{ | b_j-b_i |^{\alpha +1} }
\]
using Lemma \ref{A--pre.est.12}, (\ref{bi-Lip}), and (\ref{tsgeqbjbi}).
Then, we get
\[
	A_{i,j} + B_{i,j}
	\leq
	C_{\text g}
	L^{2- \alpha} \frac 1 { n^{ 3-( \alpha +1)/6q } }
\]
using (\ref{>Nn--est.}), (\ref{bi-Lip}), and (\ref{div.bk}).
Therefore, we obtain
\[
	\sum_{ i=1 }^n \sum_{ | j-i | > N_n }
	\max\left\{ 
	\varepsilon_n^{ \frac{ 1 }{ 4q } }
	,
	n^{ - \frac{ 1 }{ 6q } }
	\right\}^{ - \alpha ( q-1 ) } 
	( A_{i,j} + B_{i,j} )
	\leq
	C_{ \text g } L^{ 2 - \alpha }
	\frac{ 1 }
	{ 
	n^{ 1 - (\alpha+1) / (6q) }
	}.
\]
\end{proof}

Using (\ref{>Nn--est.}), we obtain
\[
	C_{i,j}
	\leq
	C_{\text g} L^{2-\alpha} \frac 1 {n^{ \alpha +2-2/q }}
	\varepsilon_n^{(4- \alpha)/(4q)},\ \ 
	D_{i,j}
	\leq
	C_{\text g} L^2 \frac 1 { n^{\alpha +2-2/q} }
	\varepsilon_n^{1/q}.
\]
Therefore, we can show the next lemma in a similar manner to the proof of Proposition \ref{AB--est.2}.

\begin{prop}\label{D--est.2}
We have
\[
	\displaystyle
	\sum_{ i=1 }^n \sum_{ | j-i | > N_n } 
	\max\left\{ 
	\varepsilon_n^{ \frac{ 1 }{ 4q } }
	,
	n^{ - \frac{ 1 }{ 6q } }
	\right\}^{ - \alpha ( q-1 ) }
	C_{i,j}
	\leq
	C_{ \rm g }
	L^{ 2 - \alpha }
	\varepsilon_n^{\ (5 \alpha q -8) / (4q)},
\]
and
\[
	\displaystyle
	\sum_{ i=1 }^n \sum_{ | j-i | > N }
	\max\left\{ 
	\varepsilon_n^{ \frac{ 1 }{ 4q } }
	,
	n^{ - \frac{ 1 }{ 6q } }
	\right\}^{ - \alpha q } 
	D_{i,j}
	\leq
	C_{\rm g}
	L^2
	\varepsilon_n^{\ (5 \alpha q -8)/(4q)}.
\]
\end{prop}

Using Propositions \ref{AB--est.2} and \ref{D--est.2},
we get
\[
	| X-Y |
	\leq
	C_{\text g}
	L^{2- \alpha q} 
	\left\{
	\frac{1}{n^{1-( \alpha q+1 )/(6q)}}
	+
	\varepsilon_n^{(5 \alpha q -8)/(4q)}
	\right\}
	\to 0
\]
as $ n \to \infty $,
and this proves Theorem \ref{inscribed}.
\qquad \qquad \qquad \qquad \qquad \qquad \qquad \quad \ \ $\square$

%
%

\section{$\Gamma$-convergence}

In this section, we prove that 
$ \mathcal{E}_n^{\alpha ,q} $
converges to
$ \mathcal{E}^{\alpha ,q} $
in the sense of $ \Gamma $-convergence.
When we consider $\Gamma$-convergence,
it is necessary that
we consider the functionals
$ \mathcal{E}_n^{\alpha ,q} $
and
$ \mathcal{E}^{\alpha ,q} $
on a common set of simply closed curves.
Hence, we need to extend their domains.


\subsection{Preparation}\label{sec.gamma}

In this subsection, we give the definition of $\Gamma$-convergence and introduce its fundamental property, and we extend the domains of
$ \mathcal{E}^{\alpha ,q} $
and
$ \mathcal{E}_n^{\alpha ,q} $.

\begin{df}[$\Gamma$-convergence]
Let
$ X $
be a metric space.
If
$ \mathscr{F}_n : X \to \overline{\mathbb{R}} $
and
$ \mathscr{F} : X \to \overline{\mathbb{R}} $
satisfy the following two properties for all $ x \in X $,
we say that
$ \mathscr{F}_n $
\textit{$\Gamma$-converges} to
$ \mathscr{F} $
on $X$
and denote this by
$
	\mathscr{F}_n \overset{\Gamma}{\longrightarrow}
	\mathscr{F}\ \mbox{on}\ X 
$.
\begin{enumerate}
	\item $(\,\liminf$ inequality$)$
	For all
	$ \{ x_n \} \subset X $
	converging to $x$
	in $X$,
	we have
	\[
		\mathscr{F} (x) \leq \liminf_{n \to \infty} \mathscr{F}_n (x_n).
	\]
	\item $(\,\limsup$ inequality$)$
	There exists
	$ \{ x_n \} \subset X $
	converging to $x$ in $X$
	and we have
	\[
		\mathscr{F} (x) \geq \limsup_{n \to \infty} \mathscr{F}_n (x_n).
	\]
\end{enumerate}
\end{df}

The following lemma states a sufficient condition under which the minimum of
$ \mathscr{F} $
is less than that of
$ \mathscr{F}_n $.
This lemma is useful for the investigation of minimality of functionals.

\begin{lem}\label{gamma--min}
Let
$ ( X , d_X )$
be a metric space,
and let
$Y$
be a subspace of
$X$.
Assume that
$ \mathscr{F}_n $,
$ \mathscr{F} : X \to \overline{\mathbb{R}} $
satisfy the following.
\begin{enumerate}
	\item
	We have
	\[
		\mathscr{F} (x) \leq \liminf_{n \to \infty} \mathscr{F}_n (x_n)
	\]
	for all
	$ \{ x_n \} \subset X $
	such that
	$ d_X (x_n , x) \to 0\ ( x \in X )$
	as $ n \to \infty $.
	\item 
	For all
	$ y \in Y $,
	there exists
	$ \{ y_n \} \subset X $
	such that
	$ d_X (y_n,y) \to 0 $
	as $ n \to \infty $
	and
	\[
		\mathscr{F} (y) \geq \limsup_{n \to \infty} \mathscr{F}_n (y_n).
	\]
\end{enumerate}
Then,
for
$ z_n $,
$ z \in X $
satisfying
\[
	d_X ( z_n , z ) \to 0,
	\ \ 
	\left|
	\mathscr{F}_n (z_n) - \inf_X \mathscr{F}_n
	\right|
	\to 0
\]
as $ n \to \infty $,
we have
\[
	\mathscr{F} (z) \leq \liminf_{n \to \infty} \inf_X \mathscr{F}_n
	\leq \inf_Y \mathscr{F}.
\]
\end{lem}

Next,
we extend the domains of
$ \mathcal{E}_n^{\alpha,q}$
and
$ \mathcal{E}^{\alpha,q} $.
For a given
tame knot class
$ \mathcal{K} $,
let
$ \mathcal{C} ( \mathcal{K} ) $
be the set of simply closed curves of length $1$ belonging to
$ \mathcal{K} $,
and let
$ \mathcal{P}_n ( \mathcal{K} ) $
be the set of equilateral polygons with $n$ edges with total length $1$ belonging to
$ \mathcal{K} $.
Also, we set
\[
	\mathcal{X}(\mathcal{K})
	:=
	\left(
	\mathcal{C} ( \mathcal{K} ) \cap
	C^1 ( \mathbb{S}_1 , \mathbb{R}^d )
	\right)
	\cup
	\bigcup_{n \in \mathbb{N}} \mathcal{P}_n ( \mathcal{K} ).
\]
Furthermore,
let $d_{L^1}$, $d_{W^{1,\infty}} : \mathcal{X}(\mathcal{K}) \times \mathcal{X}(\mathcal{K}) \to \R$
be two metric functions
induced from the $L^1$-norm or $W^{1,\infty}$-norm,
respectively.
Then, we consider a metric function
$d_X : \mathcal{X}(\mathcal{K}) \times \mathcal{X}(\mathcal{K}) \to \R$
for which there exist two constants $C_1$, $C_2>0$ such that
\begin{equation}
	C_1 d_{L^1} (f,g)
	\leq
	d_X (f,g)
	\leq
	C_2 d_{W^{1,\infty}} (f,g)
	\label{emb}
\end{equation}
for
$ f $,
$ g \in \mathcal{X}(\mathcal{K}) $.
For example, $d_X (f,g) := \| f-g \|_{L^r (\mathbb{S}_1 , \mathbb{R}^d)}$
 or
$\| f-g \|_{W^{1,r} (\mathbb{S}_1 , \mathbb{R}^d)}$ $(r \in [1,\infty])$
satisfies (\ref{emb}) because $\mathbb{S}_1$ is a bounded set.
In what follows,
we put
\[
	X := ( \mathcal{X}(\mathcal{K}) , d_X ).
\]
Moreover, let
\[
	Y := 
	\left(
	\mathcal{C} ( \mathcal{K} ) \cap
	C^1 ( \mathbb{S}_1 , \mathbb{R}^d ) \cap
	W^{1+ \sigma ,2q}  ( \mathbb{S}_1 , \mathbb{R}^d )
	,d_X
	\right).
\]
We extend the domain of
$ \mathcal{E}_n^{\alpha,q} $ to $X$
as follows.
For
$ m \ne n $,
$ p_n \in \mathcal{P}_n ( \mathcal{K} ) $,
and
a simply closed curve
$ \gamma $,
we \textit{define}
\[
	\mathcal{E}^{\alpha ,q}_m (p_n) := \infty,
	\ \mathcal{E}^{\alpha ,q}_m (\gamma) := \infty.
\]
Concerning the extension of the domain of
$ \mathcal{E}^{\alpha,q} $,
we obtain the following proposition.

\begin{prop}\label{polygon--expand}

Let
$ p_n $
be a polygon of length $1$ with $n$ edges and vertices
$ p_n (a_i) \in \mathbb{R}^d \ (i=1, \ldots ,n)$.
Suppose
$ \alpha \in ( 0, \infty ) $,
$ q \in [1, \infty ) $
with
$ 2 \leq \alpha q < 2+1/q $.
Then, we have
$ p_n \notin W^{1+\sigma ,2q} ( \mathbb{S}_1 , \mathbb{R}^d ) $,
that is,
$ \mathcal{E}^{\alpha,q} (p_n) = \infty $.

\end{prop}

\begin{proof}

It is sufficient to prove
\begin{equation}
	p_n^{\prime \prime}
	\notin W^{(\alpha-3)/2,2}(\mathbb{S}_1 , \mathbb{R}^d)
	\label{pp}
\end{equation}
for $ 2 \leq \alpha < 3 $
because we have
$
	W^{\sigma-1,2q} ( \mathbb{S}_1 , \mathbb{R}^d )
	\subset
	W^{(\alpha-3)/2 ,2} ( \mathbb{S}_1 , \mathbb{R}^d )
$.
Note that there exist constants
$ c_{j\ell} \ \ ( 1 \leq j \leq n,\ 1 \leq \ell \leq d ) $
such that
\[
	p_n^{\prime \prime}
	=
	\sum_{j=1}^n 
	\left(
	\begin{array}{c}
		c_{j1}
	\\
		\vdots
	\\
		c_{jd}
	\end{array}
	\right)
	\delta_{a_j},
\]
where $ \delta_{a_j} $ is the Dirac measure supported at $a_j$.

In order to prove (\ref{pp}),
we show
\begin{equation}
	\sum_{k \in \mathbb{Z}} |k|^{\alpha-3}
	| (p_n^{\prime \prime} )^\wedge (k) |^2
	=
	\sum_{\ell=1}^d
	\sum_{k \in \mathbb{Z}} |k|^{\alpha -3}
	\left|
	\sum_{j=1}^n
	c_{j \ell} e^{-2\pi i ka_j} 
	\right|^2
	= \infty,
	\label{series}
\end{equation}
where
$
	(p_n^{\prime \prime} )^\wedge (k) =
	{\vphantom{
	\left\langle
	p_n^{\prime \prime}, e^{-2\pi ik\cdot}
	\right\rangle
	}}_{\mathscr{D}^\prime}\! \left\langle
	p_n^{\prime \prime}, e^{-2\pi ik \cdot}
	\right\rangle_{\mathscr{D}}
$,
and
$ i = \sqrt{-1} $.
Fix $\ell = 1, \ldots , d$. Then, we have
\begin{eqnarray*}
	\lefteqn{
	\sum_{ k \in \mathbb{Z} } |k|^{\alpha -3}
	\left|
	\sum_{j=1}^n
	c_{j \ell} e^{-2\pi i ka_j} 
	\right|^2
	}
\\
	& = &
	\sum_{ k \in \mathbb{Z} } |k|^{\alpha -3} \sum_{j=1}^n |c_{j \ell}|^2
	+
	2 \sum_{ k \in \mathbb{Z} } |k|^{\alpha -3} 
	\sum_{1 \leq j_1 < j_2 \leq n}
	c_{j_1 \ell} c_{j_2 \ell} 
	\cos 2\pi k ( a_{j_2} - a_{j_1} ).
\end{eqnarray*}
It is obvious that the first term diverges to infinity,
and the second term is bounded because the infinite series
$ \sum_{k=1}^\infty k^s \cos (ka) $
converges for
$ a \in \mathbb{R} \setminus 2\pi \mathbb{Z} $
and
$ s < 0 $.
Therefore, we get (\ref{series}).
\end{proof}


\subsection{The $\Gamma$-convergence of $ \mathcal{E}_n^{\alpha ,q} $}

Note that we prove the $\liminf$ inequality with respect to $L^1$-topology
and the $\limsup$ inequality with respect to $W^{1,\infty}$-topology
because we have to consider the $\liminf$ inequality for \textit{all} polygonal sequences $\{ p_n \}$ and the $\limsup$ inequality for
\textit{a} polygonal sequence $\{ p_n \}$.

First, we prove the $ \liminf $ inequality needed for proof of the
$ \Gamma $-convergence of $\mathcal{E}_n^{\alpha ,q}$.

\begin{thm}[The $\liminf$ inequality]\label{liminf}
Let
$ \alpha \in ( 0, \infty ) $,
$ q \in [1, \infty ) $.
Assume that
$ p_n $,
$ \gamma \in \mathcal{C}(\mathcal{K}) $
satisfy
\[
	\| p_n - \gamma \|_{L^1 (\mathbb{S}_1,\mathbb{R}^d)} \to 0
\]
as $ n \to \infty $.
Then, we have
\[
	\mathcal{E}^{\alpha ,q} (\gamma)
	\leq \liminf_{n \to \infty} \mathcal{E}_n^{\alpha ,q} (p_n).
\]
\end{thm}

\begin{proof}
We may assume
$
	\liminf_{n \to \infty}
	\mathcal{E}_n^{\alpha ,q} ( p_n ) < \infty
$.
Note that
$ p_n \in \mathcal{P}_n ( \mathcal{K} ) $
by the way we extended the domain of
$ \mathcal{E}_n^{\alpha ,q} $
.
Now, there exists $\{ n_k \}_{k=1}^\infty$ such that
\[
	n_1 < n_2 < \cdots \to \infty,\ \ 
	\liminf_{n \to \infty} \mathcal{E}_n^{\alpha ,q} (p_n)
	=
	\lim_{k \to \infty} \mathcal{E}_{n_k}^{\alpha ,q} (p_n).
\]
Thus, there exists $\{ p_{n_{k(\nu)}} \}_{\nu =1}^\infty$
which is a subsequence of $\{ n_k \}_{k=1}^\infty$
such that
$ p_{n_{k(\nu)}} \to \gamma$
as $\nu \to \infty$ a.e.\,on $\mathbb{S}_1$.
It is sufficient to prove the claim for
$\{ p_{n_{k(\nu)}} \}_{\nu =1}^\infty$.

Now, we write $p_{n_{k(\nu)}}$ as $p_n$ for simplicity.
Let
$
	s,t \in 
	\left\{
	u \in \mathbb{S}_1 
	\,\left|\, 
	\lim_{n \to \infty} p_n (u) = \gamma (u)  
	\right.
	\right\} 
$,
$ s \ne t $.
For all
$ n \in \mathbb{N} $,
we can put consecutive points
$ a_1^{(n)} , \ldots , a_n^{(n)} \in \mathbb{S}_1 $
which satisfy
$ | a_{k+1}^{(n)} - a_k^{(n)} | = 1/n $
for
$ k=1, \ldots ,n $
and such that there exists
$ i_n , j_n \in \{ 1, \ldots ,n \} $
satisfying
\[
	( s,t ) \in
	[ a_{ i_n }^{(n)} , a_{ i_n+1 }^{(n)} )
	\times
	[ a_{ j_n }^{(n)} , a_{ j_n+1 }^{(n)} ).
\]
Then, we have
\[
	\sum_{\substack{i,j=1\\i \ne j}}^n
	\Manp ( a_i^{(n)} , a_j^{(n)} )^q
	\chi_{ [ a_{ i_n }^{(n)} , a_{ i_n +1 }^{(n)} )
	\times [ a_{ j_n }^{(n)} , a_{ j_n +1 }^{(n)} ) }
	( s,t ) 
	\to 
	\Mag (s,t)^q
\]
as $ n \to \infty $. 
Using Fatou's lemma,
we have
\begin{eqnarray*}
	\lefteqn{
	\mathcal{E}^{\alpha ,q} (\gamma)
	=
	\frac 1 \alpha
	\int_{\mathbb{S}_1} \int_{\mathbb{S}_1}
	\Mag (s,t)^q
	dsdt
	}
\\
	& = &
	\frac 1 \alpha
	\int_{\mathbb{S}_1} \int_{\mathbb{S}_1}
	\lim_{n \to \infty}
	\sum_{\substack{i,j=1\\i \ne j}}^n
	\Manp ( a_i^{(n)} , a_j^{(n)} )^q
	\chi_{ [ a_{i_n}^{(n)} , a_{i_n +1 }^{(n)} )
	\times [ a_{j_n}^{(n)} , a_{j_n +1}^{(n)} ) }
	( s,t )
	dsdt
\\
	& \leq &
	\liminf_{n \to \infty} \mathcal{E}_n^{\alpha ,q} (p_n)
\end{eqnarray*}
because of the definition of
$ \{ a_k^{(n)} \}_{k=1}^n $.
\end{proof}

Furthermore, by Ascoli-Arzel\`{a}'s theorem, we get the following corollary.
\begin{cor}\label{liminf--cor}
Assume that
$ p_n \in \mathcal{P}_n (\mathcal{K}) $
satisfy that
\[
	\sup_{n \in \mathbb{N} } 
	\| p_n \|_{ L^\infty (\mathbb{S}_1 , \mathbb{R}^d) } < \infty,
	\ \ 
	\sup_{n \in \mathbb{N} } \mathcal{E}^{\alpha,q}_n (p_n) < \infty.
\]
Then, there exists a subsequence
$ \{ p_{n_j} \} $
and
$ \gamma \in W^{1+\sigma ,2q} (\mathbb{S}_1 , \mathbb{R}^d) $
such that
$ \| p_{n_j} - \gamma \|_{ L^1 (\mathbb{S}_1 , \mathbb{R}^d) } \to 0 $
as $ j \to \infty $ for
$ \alpha \in ( 0, \infty ) $,
$ q \in [1, \infty ) $
with
$ 2 \leq \alpha q < 2q+1 $.
\end{cor}

The following claim is a strong version of the $\limsup$ inequality for 
$ \gamma \in W^{1+ \sigma , 2q } ( \mathbb{S}_1 , \mathbb{R}^d ) $.
We can prove it using the method of proof of \cite[Proposition 4.1]{Sch}.

\begin{thm}[A strong version of the $\limsup$ equality]\label{limsup}
Let
$ \alpha \in (0, \infty ) $
and
$ q \in [1, \infty ) $
with
$ 2 \leq \alpha q < 2q+1 $,
and let
$
	\gamma \in \mathcal{C}( \mathcal{K} )
	\cap C^1 ( \mathbb{S}_1 ,\mathbb{R}^d )
	\cap W^{1+\sigma , 2q} ( \mathbb{S}_1 , \mathbb{R}^d )
$.
Then, there exists
$ p_n \in \mathcal{P}_n ( \mathcal{K} ) $
such that
\[
	\lim_{n \to \infty}
	\| p_n - \gamma \|_{ W^{1, \infty} ( \mathbb{S}_1 , \mathbb{R}^d ) }=0,
	\ \ 
	\lim_{n \to \infty} \mathcal{E}^{\alpha ,q}_n (p_n)
	= \mathcal{E}^{\alpha ,q}( \gamma ).
\]
\end{thm}

Next,
we show that
$ \mathcal{E}_n^{\alpha ,q} $
$ \Gamma $-converges to
$ \mathcal{E}^{\alpha ,q} $
using previous results.

\begin{thm}[$ \Gamma $-convergence of $ \mathcal{E}_n^{\alpha ,q} $]\label{gamma}
Let
$ \alpha \in (0, \infty )$
and
$ q \in [1, \infty ) $
with
$ 2 \leq \alpha q < 2q+1 $.
Then, we have
\begin{equation} 
	\mathcal{E}^{\alpha,q}_n
	\overset{\Gamma}{\longrightarrow}
	\mathcal{E}^{\alpha,q}
	\ \mbox{on}\ X.
\end{equation}
\end{thm}

\begin{proof}
Put $ \gamma \in X $.
If
$ p_n \in X $
satisfies
$ d_X (p_n , \gamma ) \to 0 $
,
we have
$ \| p_n - \gamma \|_{L^1} \leq C_1^{-1} d_X (p_n , \gamma) \to 0 $.
Then, we have
\[
	\mathcal{E}^{\alpha,q} (\gamma)
	\leq 
	\liminf_{n \to \infty} \mathcal{E}^{\alpha,q}_n (p_n)
\]
using Theorem \ref{liminf}.
This implies that
$ \mathcal{E}_n^{\alpha ,q} $
satisfies the $\liminf$ inequality.

Now, we prove the $\limsup$ inequality.
The claim is obvious in the case where
$ \gamma \in X \setminus Y $.
Therefore, let
$ \gamma \in Y $.
Then, there exists
$ p_n \in \mathcal{P}_n ( \mathcal{K} ) $
such that
\begin{equation}
	\lim_{n \to \infty} d_X (p_n , \gamma) = 0,
	\ \lim_{n \to \infty} \mathcal{E}^{\alpha,q}_n (p_n)
	= \mathcal{E}^{\alpha,q} (\gamma)
	\label{gamma--rmk}
\end{equation}
by Theorem \ref{limsup} and (\ref{emb}).
In particular, we have
\[
	\mathcal{E}^{\alpha,q} (\gamma)
	\geq \limsup_{n \to \infty} \mathcal{E}^{\alpha,q}_n (p_n).
\]
\end{proof}

\begin{rem}\label{gamma--rem}
(\ref{gamma--rmk}) implies that
$ \mathcal{E}^{\alpha,q}_n $
not only
$ \Gamma $-converges to
$ \mathcal{E}^{\alpha,q} $
but also
satisfies the assumption of Lemma \ref{gamma--min}.
\end{rem}

The following corollary suggests the following: assume that a polygonal sequence has values of the discrete energy are sufficiently close to the minimum value for all numbers of vertices.
Then, this sequence converges to a curve, which is a right circle by \cite{ACFGH}.

\begin{cor}\label{minimizer1}
If
$ p_n \in \mathcal{P}_n ( \mathcal{K} ) $
and
$ \gamma \in \mathcal{C} ( \mathcal{K} ) $
satisfy
\[
	\left|
	\inf_{ \mathcal{P}_n (\mathcal{K}) } \mathcal{E}^{\alpha,q}_n
	- \mathcal{E}^{\alpha,q}_n (p_n)
	\right|
	\to 0,\ \ 
	d_X (p_n , \gamma ) \to 0,
\]
then
$ \gamma $
is the minimizer of
$ \mathcal{E}^{\alpha,q} $
in
$ \mathcal{C} ( \mathcal{K} ) $,
and we have
\[
	\lim_{n \to \infty} \mathcal{E}^{\alpha,q}_n (p_n)
	= \mathcal{E}^{\alpha,q} (\gamma).
\]
\end{cor}

%
%

\section{Minimizers of $ \mathcal{E}_n^{\alpha ,q} $}

In this section,
we consider minimizers of a generalized discrete energy using techniques of \cite{ACFGH}.
In what follows, we set
$ \Omega := \{ (x,y) \in \mathbb{R}^2 \,|\, 0 < x \leq y \} $.
\begin{thm}\label{ippan--minimizer}
Let
$ F : \Omega \to \mathbb{R} $
be a function such that, if we set
$ g_y (u) = F( \sqrt{u} ,y ) $
for
$ u \in (0 , y^2 ] $
and
$ y \in (0, 1/2) $,
then
$ g_y $
is decreasing  and convex.
For a polygon with $n$ edges with total length $1$, set
\[
	\mathcal{E}_F ( p_n ) :=
	\sum_{\substack{i,j=1\\i \ne j}}^n
	F( | p_n (a_j) - p_n (a_i) | , |a_j-a_i| )
	| p_n (a_{i+1}) - p_n (a_i) |
	| p_n (a_{j+1}) - p_n (a_j) |.
\]
Moreover,
for $0< a < b$,
set $ [ a ]_b := \min\{ a, b-a \} $.
Then, if $ p_n \in \mathcal{P}_n ( \mathcal{K} ) $, we have
\[
	\mathcal{E}_F ( p_n )
	\geq
	\frac{1}{n} \sum_{k=1}^{n-1}
	F \left(
	\frac{1}{n} \frac{ \sin ([k]_n \pi /n) }{ \sin ( \pi /n ) } , |a_k - a_0|
	\right),
\]
and the minimizers of
$ \mathcal{E}_F $
are regular polygons with $n$ edges.
\end{thm}

The proof of Theorem \ref{ippan--minimizer} makes use of the following lemma.

\begin{lem}[$\mbox{\cite[Theorem II]{Gab}}$,\ $\mbox{\cite[Lemma 7]{ACFGH}}$]\label{dm--lem}
Let
$ n \geq 4 $,
and put
$ k = 1, \ldots ,n $.
Let
$ f : \mathbb{R} \to \mathbb{R} $
be an increasing and concave function.
Then, there exists
$ c >0 $
with
$ | v_{i+1} - v_i | \leq c $
such that
\[
	\frac{1}{n} \sum_{i=1}^n f(| v_{i+k} - v_i |^2)
	\leq
	f
	\left(
	c^2 \frac{ \sin^2 ([k]_n \pi /n) }{ \sin^2 ( \pi /n ) } 
	\right)
\]
for all
$ v_1, \ldots ,v_n \in \mathbb{R}^d $
with
$ v_{n+i} = v_i $
for
$ i=1, \ldots ,n $.
Equality holds in the above inequality only
when the polygon which is made by joining
$ v_1 , \ldots ,v_n $
by segments in turn is a regular polygon with $n$ edges.
\end{lem}

\begin{proof}[\bf{Proof of Theorem \ref{ippan--minimizer}}]
Since
$ p_n $
is an equilateral polygon, we have
\[
	\mathcal{E}_F (p_n)
	=
	\frac{1}{n^2}
	\sum_{k=1}^{n-1} \sum_{i=1}^n
	F( | p_n (a_{i+k} ) - p_n (a_i) | , |a_k - a_0| ).
\]
For
$ k=1, \ldots ,n $,
set
\[
	f_k (x) 
	=
	\left\{
	\begin{array}{ll}
		- F( \sqrt{x} , | a_k - a_0 | ) & ( 0 < x < | a_k - a_0 |^2 ),
	\\
		- F( | a_k - a_0 | , | a_k - a_0 | ) & ( x \geq | a_k - a_0 |^2 ).
	\end{array}
	\right.
\]
Then,
$ f_k(x) $
is an increasing and concave function on
$ 0 < x < | a_k - a_0 |^2 $.
Hence, using Lemma \ref{dm--lem}, we have
\begin{align*}
	\frac{1}{n} \sum_{i=1}^n
	F( | p_n (a_{i+k} ) - p_n (a_i) | , |a_k - a_0| )
	= &\,
	- \frac{1}{n} \sum_{i=1}^n
	f_k ( | p_n (a_{i+k} ) - p_n (a_i) |^2 )
\\
	\geq &\,
	- f_k
	\left(
	\frac{1}{n^2} \frac{ \sin^2 ([k]_n \pi /n) }{ \sin^2 ( \pi /n ) } 
	\right),
\end{align*}
where the equality holds only when
$ p_n $
is a regular polygon with $n$ edges by the condition of equality in Lemma \ref{dm--lem}.

Let
$ g_n \in \mathcal{P}_n ( \mathcal{K} ) $
be a regular polygon with $n$ edges,
and suppose
$ 1 \leq k \leq n $.
Then, we have
\begin{align*}
	\frac{1}{n} \frac{ \sin ([k]_n \pi /n) }{ \sin ( \pi /n ) }
	= | g_n(a_k) - g_n(a_0) |
	=&\ | g_n(a_{i+k}) - g_n(a_i) |,
\\
	| a_k-a_0 | =&\ |a_{i+k} - a_i |
\end{align*}
for all
$ i=1, \ldots ,n-1 $.
Hence, we obtain
\begin{align*}
	\mathcal{E}_F (p_n)
	\geq &\,
	- \frac{1}{n} \sum_{k=1}^{n-1}
	f_k
	\left(
	\frac{1}{n^2} \frac{ \sin^2 ([k]_n \pi /n) }{ \sin^2 ( \pi /n ) } 
	\right)
\\
	= &\ 
	\frac{1}{n} \sum_{k=1}^{n-1}
	F \left(
	\frac{1}{n} \frac{ \sin ([k]_n \pi /n) }{ \sin ( \pi /n ) } , |a_k - a_0|
	\right)
	=
	\mathcal{E}_F (g_n).
\end{align*}
Therefore, minimizers of $ \mathcal{E}_F $ are regular polygons with $n$ edges.
\end{proof}

Applying Theorem \ref{ippan--minimizer} to
$ \mathcal{E}_n^{\alpha ,q} $,
we obtain the following corollary.

\begin{cor}\label{dm--est}
Let
$ \alpha \in (0, \infty ) $
and
$ q \in [1, \infty ) $.
Then, for all equilateral polygons with $n$ edges $p_n$,
we have
\[
	\mathcal{E}_n^{\alpha ,q} (p_n)
	\geq
	\frac{n^{\alpha q -1}}{\alpha}
	\sum_{k=1}^{n-1}
	\left(
	\frac{ \sin^\alpha ( \pi /n ) }{ \sin^\alpha ([k]_n \pi /n) }
	-
	\frac{1}{[k]_n^\alpha}
	\right)^q
\]
with equality if and only if
$ p_n $
is a regular polygon with $n$ edges.
\end{cor}

\begin{proof}
For
$ (x,y) \in \Omega $,
set
\[
	F(x,y) :=
	\left(
	\frac{1}{x^\alpha} - \frac{1}{y^\alpha}
	\right)^q.
\]
Then, we have
$ F( \sqrt{u} , y ) $
is decreasing and convex on
$ u \in (0,y^2] $
whenever
$ y \in (0,1/2) $.
Therefore,
$ F $
satisfies the assumption of Theorem \ref{ippan--minimizer}.

Using Theorem \ref{ippan--minimizer},
we obtain
\[
	\mathcal{E}_n^{\alpha,q} (p_n)
	\geq 
	\frac{n^{\alpha q -1}}{\alpha}
	\sum_{k=1}^{n-1}
	\left(
	\frac{ \sin^\alpha ( \pi /n ) }{ \sin^\alpha ([k]_n \pi /n) }
	-
	\frac{1}{[k]_n^\alpha}
	\right)^q
\]
for all equilateral polygons
$ p_n $
with $n$ edges.
By the condition of equality in Lemma \ref{dm--lem},
equality holds in the above inequality only when
$ p_n $
is a regular polygon with $n$ edges.
\end{proof}

By Corollary \ref{dm--est},
we obtain the following claim about the minimizers of
$ \mathcal{E}_n^{\alpha,q} $.

\begin{thm}[Minimizers of $ \mathcal{E}_n^{\alpha ,q} $]
\label{discrete--minimizer}
Let
$ \alpha \in (0, \infty ) $
and
$ q \in [1, \infty ) $.
Then, minimizers of
$ \mathcal{E}_n^{\alpha ,q} $
in the set of equilateral polygons with $n$ edges are regular polygons.
Especially, a regular polygon with $n$ edges is the only minimizer except for congruent transformations and similar transformations.
\end{thm}

From Theorem \ref{discrete--minimizer},
we immediately obtained the following property of minimizers of
$\mathcal{E}^{\alpha ,q}_n$.

\begin{cor}\label{minimizer--circle}
Let
$ \alpha \in (0, \infty ) $
and
$ q \in [1, \infty ) $,
and let
$ p_n $
satisfy
$
	\mathcal{E}^{\alpha,q}_n (p_n)
	= \inf_{\mathcal{P}_n (\mathcal{K})} \mathcal{E}^{\alpha,q}_n
$.
Then, there exists a similar transformation such that
$ \{ p_n \} $
converges to a right circle in the sense of
$ W^{1, \infty} $
as
$ n \to \infty $.
\end{cor}

\end{document}